\documentclass{article}
\pdfoutput=1
\usepackage{epstopdf}
\usepackage{graphicx, amssymb, amsmath}
\usepackage{amsthm}
\usepackage{mathrsfs}
\usepackage{color}
\usepackage{caption}
\usepackage{algorithm}
\usepackage{algorithmic}
\usepackage{upgreek}
\usepackage{subfigure}
\usepackage{color,soul}
\usepackage[colorlinks,
            linkcolor=blue,
            anchorcolor=blue,
            citecolor=blue]{hyperref}
\usepackage{subfigure}
\usepackage{setspace}
\sethlcolor{blue}
\usepackage{float}
\usepackage{natbib}
\usepackage{diagbox}
\setcitestyle{numbers,square}

\def\k{\kern .5em}
\def\er{\kern .2em}

\begin{document}

\date{}
\author{}
\newcommand{\be}{\begin{equation}}
\newcommand{\ee}{\end{equation}}
\newcommand{\ba}{\begin{array}}
\newcommand{\ea}{\end{array}}
\newcommand{\beas}{\begin{eqnarray*}}
\newcommand{\eeas}{\end{eqnarray*}}
\newcommand{\bea}{\begin{eqnarray}}
\newcommand{\eea}{\end{eqnarray}}
\newcommand{\ome}{\Omega}

\newtheorem{theorem}{Theorem}[section]
\newtheorem{lemma}{Lemma}[section]
\newtheorem{remark}{Remark}[section]
\newtheorem{proposition}{Proposition}[section]
\newtheorem{definition}{Definition}[section]
\newtheorem{corollary}{Corollary}[section]

\newtheorem{theo}{Theorem}[section]
\newtheorem{lemm}{Lemma}[section]
\newcommand{\blem}{\begin{lemma}}
\newcommand{\elem}{\end{lemma}}
\newcommand{\bthe}{\begin{theorem}}
\newcommand{\ethe}{\end{theorem}}
\newtheorem{prop}{Proposition}[section]
\newcommand{\bprop}{\begin{proposition}}
\newcommand{\eprop}{\end{proposition}}
\newtheorem{defi}{Definition}[section]
\newtheorem{coro}{Corollary}[section]
\newtheorem{algo}{Algorithm}[section]
\newtheorem{rema}{Remark}[section]
\newtheorem{property}{Property}[section]
\newtheorem{assu}{Assumption}[section]
\newtheorem{exam}{Example}[section]

\renewcommand{\theequation}{\arabic{section}.\arabic{equation}}
\renewcommand{\thetheorem}{\arabic{section}.\arabic{theorem}}
\renewcommand{\thelemma}{\arabic{section}.\arabic{lemma}}
\renewcommand{\theproposition}{\arabic{section}.\arabic{proposition}}
\renewcommand{\thedefinition}{\arabic{section}.\arabic{definition}}
\renewcommand{\thecorollary}{\arabic{section}.\arabic{corollary}}
\renewcommand{\thealgorithm}{\arabic{section}.\arabic{algorithm}}
\newcommand{\lan}{\langle}
\newcommand{\curl}{{\bf curl \;}}
\newcommand{\rot}{{\rm curl}}
\newcommand{\grad}{{\bf grad \;}}
\newcommand{\dvg}{{\rm div \,}}
\newcommand{\ran}{\rangle}
\newcommand{\bR}{\mbox{\bf R}}
\newcommand{\bRn}{{\bf R}^3}
\newcommand{\Coinf}{C_0^{\infty}}
\newcommand{\disp}{\displaystyle}
\newcommand{\ra}{\rightarrow}
\newcommand{\Ra}{\Rightarrow}
\newcommand{\ud}{u_{\delta}}
\newcommand{\Ed}{E_{\delta}}
\newcommand{\Hd}{H_{\delta}}
\newcommand\varep{\varepsilon}
\newcommand{\RNum}[1]{\uppercase\expandafter{\romannumeral #1\relax}}
\newcommand{\tabincell}[2]{\begin{tabular}{@{}#1@{}}#2\end{tabular}}
\title{The Existence the Solution of Nonlinear Discrete Schemes and Convergence of a Linearized Iterative Method for time-dependent PNP Equations}
\author{*
	\and *
	\and *
}
\author{Yang Liu $^{1}$
	\and Shi Shu $^{2,*}$
	\and Ying Yang $^{3}$
}
\footnotetext[1]
{School of Mathematics and Computational Science, Xiangtan University,
	Xiangtan, 411105, P.R. China. E-mail: liuyang@smail.xtu.edu.cn
}

\footnotetext[2]
{$^{*}$\textbf{Corresponding author.} School of Mathematics and Computational Science, Xiangtan University,
	Xiangtan, 411105, P.R. China. E-mail: shushi@xtu.edu.cn }

\footnotetext[3]
{School of  Mathematics and Computational Science, Guangxi Colleges and Universities Key Laboratory of Data Analysis and Computation, Guangxi Applied Mathematics Center (GUET), Guilin University of Electronic Technology, Guilin, 541004, Guangxi, P.R. China. E-mail: yangying@lsec.cc.ac.cn }
%
%

\maketitle

\noindent
{\bf Abstract}\quad 
We establish the existence theory of several commonly used finite element (FE) nonlinear fully discrete solutions, and the convergence theory of a linearized iteration. First, it is shown for standard FE, SUPG and edge-averaged method respectively that the stiffness matrix is a column M-matrix under certain conditions, and then the existence theory of these three FE nonlinear fully discrete solutions is presented by using Brouwer's fixed point theorem.
Second, the contraction of a commonly used linearized iterative method--Gummel iteration is proven, and then the convergence theory is established for the iteration. At last, a numerical experiment is shown to verifies the theories.

\noindent
{\bf Keywords}:  Poisson-Nernst-Planck equations, finite element method, Gummel iteration, existence theory and convergence theory, SUPG method, edge-averaged method

\noindent
{\bf AMS(2000) subject classifications}\quad 65N15, 65N30.

\section{Introduction}\label{sec1}
\noindent

\vskip 0.3cm
The Poisson-Nernst-Planck (PNP) equations, which are coupled by the Poisson equation and the Nernst-Planck equation, were first proposed by Nernst \cite{nernst1889elektromotorische} and Planck \cite{planck1890ueber}. They are often used to describe the ion mass conservation and electrostatic diffusion reaction process, and have been widely used in the numerical simulations of biological ion channel \cite{tu2013parallel,wang2021stabilized,eisenberg1998ionic,singer2009poisson}, semiconductor devices \cite{markowich1985stationary,van1950theory,brezzi2005discretization,miller1999application}, and nanopore systems \cite{xu2018time,daiguji2004ion}.

Due to the strong coupling and nonlinearity, it is difficult to find the analytic solution except a few cases. The finite element (FE) method has been applied to solve PNP equations and is popular since it is flexible and adaptable in dealing with the irregular interface. Comparing with the plenty of work in the FE computation of PNP equations (see e.g. \cite{J.W.Jerome1991, lu2007electrodiffusion, lu2010poisson, xie2020effective,  zhang2022class}) , the theoretical analysis of FE method seems limited, especially for the existence of the discrete solution. Prohl and Schmuck \cite{prohl2009convergent} propose two classes of FE schemes for the time-dependent PNP equations and show the existence and uniqueness, and convergence of the FE solutions. In recently years, some work on error analyis has been appeared for the FE solution of PNP equations. For example, Yang and Lu  \cite{yang2013error} presented some error bounds for a piecewise FE approximation to the steady-state practical PNP problems and gave serveral numerical examples including biomolecular problems to support the analyisis. The a priori error estimates of both semi- and fully discrete FE approximation for time-dependent PNP equations are presented in \cite{sun2016error}, in which the optimal convergence order in $L^{\infty}(H^1)$ and $L^2(H^1)$ norms are obtained with a linear element FE discretization. The error estimates in $H^1$ norms for the FE approximation to the nonlinear PNP equations are shown in \cite{yang2020superconvergent}. The superconvergence results are also presented for this model by using the gradient recovery technique, which are successfully applied to improve the efficiency of the Gummel iteration for a practical ion channel problem. Recently, a generilized FE method--virtual element method is applied to solve PNP equations on arbitrary polygons or polyhedrons, the error estimates of which are shown in \cite{liu22, degh23}.
\vskip 0.3cm
There exist problems such as the poor stability or the bad approximation when using standard FE method to discretize pratical PNP equations. Therefore, some improved FE methods have emerged. The SUPG method is applied  to a class of modified PNP equations in ion channel in \cite{tu2015stabilized} to improve the robustness and stability of the standard FE algorithm. A new stable FE method--SUPG-IP was proposed in \cite{SUPG-IP} for the steady-state PNP equations in ion channel, which has better robustness than the standard~FE and~SUPG methods. The inverse average finite element method was constructed for a class of steady-state PNP equations in nanopore systems in \cite{zhang2021inverse}, which effectively solves the problems of non physical pseudo oscillations caused by convection dominance. In \cite{zhang2022class}, four FE methods with averaging techniques were used to discrete
steady-state PNP equations in semiconductor, which is more stable than standard FE method.
Although these improved FE method have show good efficiency in the computation of PNP equations, the theoretical analysis for the discrete solution is very limited. To our knowedege, there is no theoretical result for the existence of the discrete solution for the improved FE method.


In this paper, one of the main contributions is that the existence theory of the solution to a type of FE nonlinear discrete system
is established for the time-dependent PNP equations. The type of FE schemes includes three commonly used one: standard FE, SUPG, and edge-averaged finite element (EAFE) schemes, the nonlinear discrete systems of which have a unified expression form. We strictly prove that the coefficient matrix of the unified discrete system is a column M-matrix under certain weak conditions.  
To establish the existence theory of the solution, a suitable compact convex set needs to be careful constructed, and Brouwer fixed point theorem and mathematical induction are also applied.

\vskip 0.3cm
The above nonlinear discrete systems are usually linearized by iterative methods such as~Gummel or~Newton iteration. For example, in \cite{mathur2009multigrid}, the Newton linearization is used for the adaptive finite volume discrete system of PNP equation on unstructured grid.
 Xie and Lu \cite{xie2020effective} use the Slotboom transformation to transform the PNP equations with periodic boundary conditions into equivalent equations, and give an acceleration method for the Gummel iteration. In this paper, we present the Gummel iteration for the time-dependent PNP equations based on the EAFE scheme. The contraction and convergence theories are established for this iteration, which is another main constribution in this paper. The results of the contraction and convergence theories can be easily generalized to other linearized iterations such as Newton iteration. Since PNP equations are a strong nonlinear coupled system, the analysis to the solution of iteration needs careful treatment of the special nonlinear term.

 The rest of this paper is organized as follows. In Section 2, we introduce the time-dependent PNP equations and three commonly used FE descrete schemes. In Section 3, first we present the corresponding nonlinear discrete algebraic system for the three FE schems. Then, we show the existence of the solutions to the three FE nonlinear fully discrete schemes. After that, we present the Gummel iteration combining with EAFE scheme and estiblish the contraction and convergence theories for the iteration. A numerical example is also reported in this section to verifies the contraction and convergence theories for the Gummel iteration. Finally, some conclusions are made in Section 4.

 \vskip 0.3cm

\setcounter{lemma}{0}
\setcounter{theorem}{0}
\setcounter{corollary}{0}
\setcounter{equation}{0}
\setcounter{remark}{0}
\section{ The continuous and discrete problems}
Let $\Omega \subset \mathbb{R}^{d}$ $(d=2,3)$ be a bounded Lipschitz domain. We adopt the standard notations for Sobolev spaces $W^{s,p}(\Omega)$ and
their associated norms and seminorms. For $p=2$, denote by
$H^{s}(\Omega)=W^{s,2}(\Omega)$ and $H_0^1(\Omega)=\{v\in H^{1}(\Omega):v|_{\partial\Omega}=0\}$. For simplicity, let $\|\cdot\|_{s}=\|\cdot\|_{W^{s,2}(\Omega)}$ and $\|\cdot\|=\|\cdot\|_{L^2(\Omega)}$. We use $(\cdot,\cdot)$ to denote the standard $L^{2}$-inner product.
\vspace{3mm}
\subsection{PNP equations}

 Consider the following time-dependent PNP equations (cf. \cite{shen2020decoupling})
 \begin{equation}\label{pnp equation}
 \left\{\begin{array}{lr}
 -\Delta\phi-\sum\limits^{2}_{i=1}q^{i}p^{i}=f, & \text { in } \Omega,~\text{for}~t\in(0, T], \vspace{1mm}\\
 {\partial_t p^{i}}-\nabla\cdot(\nabla{p^{i}}+ q^{i}p^{i}\nabla\phi)=F^i, & \text { in } \Omega,~\text{for}~t\in(0, T], i=1,2,\vspace{1mm}
 \end{array}\right.
 \end{equation}
 with the homogeneous Dirichlet boundary conditions
 \begin{equation}\label{boundary}
 \left\{
 \begin{array}{l}
 \;\phi=0,\;\text{on}\;\partial\Omega,~\text{for}~t\in(0, T],\vspace{0.5mm}\\
 \;p^i=0,\;\text{on}\;\partial\Omega,~\text{for}~t\in(0, T],  i=1,2,
 \end{array}
 \right.
 \end{equation}
 where  $p^i,~i=1,2$ represents the concentration of the $i$-th ionic species, $p_t^{i}=\frac{\partial p^i}{\partial t}$, $\phi$ denotes the electrostatic potential, the constant $q^i$ is the charge of the species $i$, and $f$ and $F^{i}$ are the reaction source terms. Denote the initial concentrations and potential by $p^{i}_0,~\phi_{0},~i=1,2$.


 The weak formulation of (\ref{pnp equation})-(\ref{boundary}) is that:
 find $p^i\in L^2(0,T;H_0^1(\Omega))\cap L^{\infty}(0,T;L^{\infty}(\Omega))$, $i=1,2$, and $\phi(t)\in H_0^1(\Omega)$ such that
 \begin{eqnarray}
 \label{weak1con}
 (\nabla\phi,\nabla w)-\sum\limits^{2}_{i=1}q^{i}(p^{i},w)=(f,w),~~\forall w\in H_0^1(\Omega), \\
 \label{weak2con}
 (\partial_tp^i,v)+(\nabla{p^{i}},\nabla v)+ (q^{i}p^{i}\nabla\phi,
 \nabla v)=(F^i,v), ~~\forall v\in H_0^1(\Omega), ~i=1,2.
 \end{eqnarray}

  The existence and uniqueness of the solutious to (\ref{weak1con})-(\ref{weak2con}) have been presented in \cite{gajewski1986basic} for $F^i=R(p^1,p^2)=r(p^1,p^2)(1-p^1p^2)$.
  	Here $p^1,~p^2$ represent the densities of mobile holes and electrons respectively in a semiconductor device and $R(p^1,p^2)$ is the net recombination rate (see \cite{gajewski1986basic}). The function $r: R_+^2\rightarrow R_+$ are required to be Lipschitzian and $p^1,~p^2\in W^{0,\infty}(\Omega)$.

 \subsection{Three discrete schemes}
In this subsection, we introduce three commonly used FE schemes for PNP equations including standard FE, SUPG  and EAFE schemes.

 Suppose $\mathcal{T}_h = \{ K \}$ is a  partition of $\Omega$, where $K$ is the element and $h=\max\limits_{K\in\mathcal{T}_h} \{h_K\}$, $h_K=\mbox{diam}~K$. Define the linear finite element space as follows
 \begin{eqnarray}
 V_h =\{v \in H^1(\Omega): v|_{\partial\Omega}=0~\mbox{and}~v|_K \in \mathcal{P}_1(K), ~~\forall K\in \mathcal{T}_h\},
 \end{eqnarray}
 where $\mathcal{P}_1(K)$
 denotes the set of all polynomials with the degree no more then $1$ on the element $K$.  Denote the set of basis function vector in $V_h$ by
 \begin{eqnarray}\label{Psi}
 \Psi=(\psi_{1},\ldots,\psi_{{n_h}}),
 \end{eqnarray}
where ${n_h}=\dim(V_h )$.

  The standard semi-discrete finite element formulation corresponding to \eqref{weak1con}-\eqref{weak2con} is as follows: find $p^{i}_h,~i=1,2$ and $\phi_h\in V_h$, such that
 \begin{eqnarray}
 \label{semiweak2}
 (\nabla\phi_h,\nabla w_h)-\sum\limits^{2}_{i=1}q^{i}(p^{i}_h,w_h)=(f,w_h),~~\forall w_h\in V_h, \\
 \label{semiweak1}
(\partial_tp^i_h,v_h)+(\nabla{p^{i}_h},\nabla v_h)+ (q^{i}p^{i}_h\nabla\phi_h,
 \nabla v_h)=(F^i,v_h), ~~\forall v_h\in V_h,~i=1,2.
 \end{eqnarray}

 In order to present the full discretization of (\ref{pnp equation})-\eqref{boundary},
 define a partition $0<t^0<t^1<\cdots t^N=T$ with time step $\tau=\max\{t^n-t^{n-1},n=1,2,\cdots N\}$. Also for any $u$, denote by
 \begin{align*}
 u^n({\bf x})=u({\bf x},t^n)
 \end{align*}
 and
 $$D_\tau u^{n+1} =\frac{u^{n+1}-u^{n}} { \tau} , ~\mbox{for}~ n=0,1,2,\cdots,N-1.$$
Next three common used linear finite element discretizations are introduced. First, the fully discrete FE approximation for  (\ref{pnp equation})-\eqref{boundary} is: find
$p_h^{i,n+1}\in V_h,~i=1,2$ and $\phi_h^{n+1}\in V_h$, such that
{\small
\begin{align} \label{fulldis2-FEM}
 (\nabla\phi_h^{n+1},\nabla w_h)-\sum\limits_{i = 1}^2 q^i(p_h^{i,n+1},w_h)
 &= (f^{n+1 },w_h),  \forall w_h\in V_h,\\\label{fulldis-FEM}
 (D_\tau p_h^{i,n+1},v_h)+(\nabla p_h^{i,n+1},\nabla v_h)+(q^ip_h^{i,n+1}\nabla\phi^{n+1}_h,\nabla v_h)
 &= (F^{i,n+1},v_h), \forall v_h\in V_h,
 \end{align}}
where $f^{n+1 }=f (t^{n+1},\cdot)  $, $F^{i,n+1}=F^i(t^{n+1},\cdot),~i=1,2$.

Note that the accurate solution to convection dominated equations often has internal or exponential boundary layers, and the SUPG scheme is a commonly used numerical method to overcome numerical oscillations caused by the boundary layers. The fully discrete SUPG scheme for (\ref{pnp equation})-\eqref{boundary} is as follows:
find $p_h^{i,n+1}\in V_h,~i=1,2$ and~$\phi_h^{n+1}\in V_h$, such that
{
	\begin{eqnarray} \label{fulldis-SUPG-temp}
	&& (\nabla\phi^{n+1}_h,\nabla w_h)-\sum\limits_{i = 1}^2 q^i(p_h^{i,n+1},w_h)
	= (f^{n+1 },w_h),~\forall w_h\in V_h ,
	\\ \notag
	&&(D_\tau p_h^{i,n+1},v_h)+(\nabla{p_h^{i,n+1}},\nabla v)+ (q^{i}p_h^{i,n+1}\nabla\phi_h^{n+1},
	\nabla v_h) \\&+&\sum\limits_{K\in \mathcal T_{h }}(-\nabla\cdot(\nabla{p_h^{i,n+1}}
	+ q^{i}p_h^{i,n+1}\nabla\phi_h^{n+1}),-q^i C_K(x)  \nabla {\phi}_h^{n+1} \cdot \nabla v_h )_K\notag\\&=&(F^{i,n+1},v_h)+\sum\limits_{K\in \mathcal T_{h}}(F^{i,n+1}-D_\tau p_h^{i,n+1},-q^i C_K(x) \nabla {\phi}_h^{n+1}\cdot \nabla v_h  )_K,~\forall v_h\in V_h,
	\label{fulldis2-SUPG-temp}
	\end{eqnarray}}
where
$C_K(x)=\left\{\begin{array}{ll}
\frac{\tilde \tau h_K }{2\| {q^i\nabla \phi^{n+1}_h}\|_{L^{\infty}( K)}}, & \mbox { if} ~P_{K} \geq 1, ~\\
\frac{\tilde \tau h_K ^{2}}{4  }, & \mbox {if } ~P_{K}<1, ~
\end{array}\right. \forall  x \in K$,
with $P_{K}=\frac{h_K \| {q^i\nabla \phi_h^{n+1}}\|_{L^{\infty}(K)}}{2 }$. Here {
$\sum\limits_{K\in \mathcal T_{h }}(-\nabla\cdot(\nabla{p_h^{i,n+1}}+ q^{i}p_h^{i,n+1}\nabla\phi^{n+1}_h),-q^i C_K(x)  \nabla {\phi}_h^{n+1} \cdot \nabla v_h )_K$} and {
$\sum\limits_{K\in \mathcal T_{h}}(F^{i,n+1}-D_\tau p_h^{i,n+1},-q^i C_K(x) \nabla {\phi}_h^{n+1}\cdot \nabla v_h  )_K$} are stabilization term. Since $\nabla\cdot (\nabla p_h ^{i,n+1})=0$ and $ \nabla\cdot(\nabla\phi_h^{n+1} )=0$ for linear element discretization, the equations ~\eqref{fulldis-SUPG-temp}-\eqref{fulldis2-SUPG-temp} can be written as
{\small
	\begin{eqnarray} \label{fulldis-SUPG}
	&& (\nabla\phi_h^{n+1},\nabla w_h)-\sum\limits_{i = 1}^2 q^i(p_h^{i,n+1},w_h)
	=   (f^{n+1},w_h),~\forall w_h\in V_h ,\\\notag
	&&(D_\tau p_h^{i,n+1},v_h)+(\nabla{p_h^{i,n+1}},\nabla v)+ (q^{i}p_h^{i,n+1}\nabla\phi_h^{n+1},
	\nabla v_h)+\sum\limits_{K\in \mathcal T_{h }}(-  q^{i}\nabla p_h^{i,n+1}\cdot \nabla\phi_h^{n+1} ,-q^i C_K(x)  \nabla {\phi}^{n+1}_h \cdot \nabla v_h )_K\\&=&(F^{i,n+1},v_h)+\sum\limits_{K\in \mathcal T_{h}}(F^{i,n+1}-D_\tau p_h^{i,n+1},-q^i C_K(x) \nabla {\phi}^{n+1}_h\cdot \nabla v_h  )_K,~\forall v_h\in V_h.\label{fulldis2-SUPG}
	\end{eqnarray}}


 Another commonly used method to deal with
 dominated convection is EAFE method. It has been used to solve NP equations (cf. \cite{zhang2021inverse}).
  In order to present the EAFE scheme for PNP equations, for any element $K$ with the given number $k$, suppose $E$ is the edge with the endpoints $x_{k_\nu}$ and $x_{k_\mu}$, where ${k_\nu}$ is the whole number corresponding to the local number $\nu$. Let $\tau_E= x_{k_\nu}-x_{k_\mu},~{k_\nu} <{k_\mu} $ for any $E$ in $\mathcal T_h$.


 The EAFE fully discrete scheme for PNP equations are as follows: find
 $p_h^{i,n+1}\in V_h ,~ i=1,2$ and $\phi_h^{n+1}\in V_h $, satisfying
{\small \begin{align}
 \label{fulldis2-EAFE}
(\nabla\phi_h^{n+1},\nabla w_h)&-\sum\limits_{i = 1}^2 q^i(p_h^{i,n+1},w_h)
=   (f^{n+1 },w_h),~\forall w_h\in V_h ,\\
(D_\tau p_h^{i,n+1},v_h)+ & \sum\limits_{K\in \mathcal{T}_{h}}[\sum\limits_{E\subset K}\omega_E^K\tilde{\alpha}^{K,i}_{E}(\phi_h^{n+1})\delta_E(e^{q^i\phi_h^{n+1}}p_h^{i,n+1})\delta_Ev_{h}]
 =  \sum\limits_{K\in \mathcal{T}_{h}}(F^{i,n+1},v_h)_K,  ~\forall v_h\in V_h , \label{fulldis-EAFE}
 \end{align}}
 where
{\small\begin{equation}\label{omega}
 	\tilde{\alpha}^{K,i}_{E }(\phi_h )=\left(\frac{1}{|\tau_{E }|}\int_{E } e^{q^i\phi _h}\mathrm{d}s\right)^{-1},~
 	\omega_{E }^K  :=\omega_{E_{\nu \mu }}^K  =-(\nabla \psi_{k_\mu},\nabla \psi_{k_\nu})_K ,~\delta_{E }( u)=  u_{k_\nu}- u_{k_\mu}, ~ u=e^{q^i\phi_h }p_h^{i}.
\end{equation}}

We have introduced PNP equations and three commonly used FE discretizations. In next section, we present the existence of the solutions to the three FE schemes and the converence analysis for a linearized iteration.

 \setcounter{lemma}{0}
 \setcounter{theorem}{0}
 \setcounter{corollary}{0}
 \setcounter{equation}{0}
  \section{The existence of the discrete solution and the convergence analysis of the Gummel iteration }
 In this section, first we give the nonlinear algebraic system for the fully discrete FE schemes, then establish the theory of the existence of the solution to the schemes. After that, we present the contraction and convergence theories of a linearized iteration.

 \subsection{The existence of the fully discrete solution}


 Note that the linear finite element basis function $\phi_h$ on element $K$ can be expressed as
 \begin{equation}\label{phihEi}
\phi_h|_E =\sum_{m=1}^4\Phi_{k_m}\lambda^K_m= \lambda^K  \Phi_K,
\end{equation}
 where $ \lambda^K =(\lambda^K_1,\lambda^K_2,\lambda^K_3,\lambda^K_4)$ and $ \Phi_K=(\Phi_{k_1},\Phi_{k_2},\Phi_{k_3},\Phi_{k_4})^T $ are the volume coordinate vector and degree of freedom vector of element K.

 Assume $n_h$ dimensional vectors
 \begin{equation}\label{def-PhiG3}
 \Phi^{n+1}=(\Phi^{n+1}_1,\ldots,\Phi^{n+1}_{n_h})^T,\quad G_\Phi^{n+1}=((f^{n+1 },\psi_1),\ldots,(f^{n+1 },\psi_{n_h}))^T,
 \end{equation}
 and $2n_h$ dimensional vectors
 \begin{equation}\label{def-Pn}
P^\mu=
\begin{pmatrix}
 P^{1,\mu}\\
P^{2,\mu}
\end{pmatrix}
,~~
P^{i,\mu}=(p^{i,\mu}_1,\ldots,p^{i,\mu}_{n_h})^T,~\mu=n,n+1,~ i=1,2,
\end{equation}
{\small
\begin{equation}\label{def-PG}
G^{n+1}=
\begin{pmatrix}
 G^{1,n+1}\\
G^{2,n+1}
\end{pmatrix},~ G^{i ,n+1}=(g^{i ,n+1}_1,\ldots,g^{i ,n+1}_{n_h})=((F^{i ,n+1},\psi_1),\ldots,(F^{i,n+1},\psi_{n_h}))^T,~ i=1,2,
\end{equation}}
$n_h\times 2n_h$ lumped mass matrix
\begin{equation}  \label{C4-M-def}
\bar M =-(q^1M,q^2 M), ~ M=\frac{1}{4}\textrm{diag}(|\Omega_1|,\cdots ,|\Omega_{n_h}|), ~ \Omega_k=supp(\psi_k),~k=1,\ldots,{n_h},
\end{equation}
and $n_h\times n_h$ stiff matrix
\begin{equation}  \label{C4-AL-def}
A_L=(  \nabla\Psi^T,\nabla \Psi).
\end{equation}
Then the corresponding nonlinear algebraic equation for \eqref{fulldis2-FEM}-\eqref{fulldis-FEM} is
\begin{eqnarray}\label{FEM-non}
A_h^{1}(U_h^{1})U_h^{1} =F_h^1.
\end{eqnarray}
Here the coefficient matrix $A_h^{1}(U_h^{1})$, solution vector~$U_h^{1} $ and right hand vector~$F_h^1$ are respectively as follows
\begin{eqnarray}\label{FEM-non-AFU}
A_h^{1}(U_h^{1})=\left(
  \begin{array}{cc}
    A_L& \bar M \\
   \bf 0    &  {\bf M} + \tau \bar A^1(\Phi^{n+1}) \\
  \end{array}
\right)
,~U_h^{1}=\left(
  \begin{array}{c}
     \Phi^{n+1,1}\\P^{ n+1,1}
  \end{array}
\right),~F_h^{1}=\left(
  \begin{array}{c}
    G_\Phi^{n+1} \\
    F^{ n,1}
  \end{array}
\right),
\end{eqnarray}
where~$2n_h\times 2n_h$ lumped mass matrix and stiff matrix are respectively given by
\begin{equation}  \label{C4-M-def-2}
 {\bf M} =diag(M, M), ~\bar A^1(\Phi^{n+1})=diag(A_L+q^1C(\Phi), A_L+q^2C(\Phi)),
\end{equation}
right hand vector
\begin{eqnarray}\label{def-F}
F^{n,1}=
\left(\begin{array}{c}
    F^{1,n,1} \\
    F^{2,n,1}
  \end{array}\right)= \tau G^{n+1}+ {\bf M} P^n
 =\tau
\left(\begin{array}{c}
    G^{1,n+1}+MP^{1,n} \\
    G^{2,n+1}+MP^{2,n}
  \end{array}\right),
\end{eqnarray}
~$\bar M$, $M$ and~$A_L$ are defined by ~\eqref{C4-M-def} and \eqref{C4-AL-def}, respectively, and the general element of $C(\Phi)=(C(\Phi)_{ij})_{n_h\times n_h}$ is $C(\Phi)_{ij}=(\psi_j\nabla\phi_h ,\nabla \psi_i) $.

The nonlinear algebraic equation system for the SUPG discrete system \eqref{fulldis-SUPG}-\eqref{fulldis2-SUPG} is given by
{
\begin{eqnarray}\label{SUPG-non}
A_h^{2}(U_h^{2})U_h^{2} =F_h^2,
\end{eqnarray}}
where
\begin{eqnarray}\label{SUPG-AU}
A_h^{2}(U_h^{2})=\left(
  \begin{array}{cc}
    A_L& \bar M \\
   \bf 0    &  {\bf M} + \tau \bar A^2(\Phi^{n+1})  \\
  \end{array}
\right)
,~U_h^{2}=\left(
  \begin{array}{c}
     \Phi^{ n+1,2}\\P^{ n+1,2}
  \end{array}
\right), ~F_h^{2}=\left(
  \begin{array}{c}
    G_\Phi^{n+1} \\
    F^{n,2}
  \end{array}
\right).
\end{eqnarray}
Here
\begin{eqnarray}\label{SUPG-AF}
 \bar A^2(\Phi^{n+1}) =\bar A^1(\Phi^{n+1})+A^{stable}(\Phi^{n+1}) ,~ F^{n,2}=F^{n,1}+F ^{stable},
\end{eqnarray}
and~$ \bar M$, ${\bf M}$ and ~$A_L$ are defined by~\eqref{C4-M-def}, \eqref{C4-M-def-2} and~\eqref{C4-AL-def}, respectively,  the submatrix of the FEM system~$\bar{A}^1(\Phi^{n+1})$ and ~$F^{n,1}$ are given by~\eqref{C4-M-def-2} and~\eqref{def-F}, respectively, $A^{stable}(\Phi^{n+1})$ and~$F^{stable}$ are stiffness matrix and right hand vector for the stabilization term.

It is easy to know the nonlinear algebraic equation for the EAFE discrete system ~\eqref{fulldis2-EAFE}-\eqref{fulldis-EAFE} is as  follows:
{
\begin{eqnarray}\label{EAFE-non}
A_h^{3}(U_h^{3})U_h^{3} =F_h^3.
\end{eqnarray}}
Here the coefficient matrix, solution vector and right hand vector
\begin{eqnarray}\label{EAFE-non-AUF}
A_h^{3}(U_h^{3})=\left(
  \begin{array}{cc}
    A_L& \bar M \\
   \bf 0    &  {\bf M} + \tau \bar A^3(\Phi^{n+1}) \\
  \end{array}
\right)
,~U_h^{3}=\left(
  \begin{array}{c}
     \Phi^{ n+1,3}\\P^{ n+1,3}
  \end{array}
\right),~F_h^{3}=F_h^{1},
\end{eqnarray}
where $ \bar M$, ${\bf M}$ and $A_L$ are defined by \eqref{C4-M-def}, \eqref{C4-M-def-2} and \eqref{C4-AL-def}, respectively, the right hand sider $F_h^{1}$ is given by  \eqref{FEM-non-AFU},
 $2n_h\times 2n_h$ stiff matrix
\begin{equation}  \label{C4-A-def}
\bar A^3(\Phi^{n+1})=diag(\bar A^1(\Phi^{n+1}),\bar A^2(\Phi^{n+1})),
\end{equation}
the general element of the element stiffness matrix ~$\bar A^{i,K}(\Phi^{n+1})$  on tetrahedron element $K$ for $\bar A^i(\Phi^{n+1})$ is as follows:
{
\begin{equation}\label{aij}
\bar a_{\nu \mu }^{i,K}(\Phi^{n+1})=\left\{\begin{array}{ll}
 & -\omega^K_{E_{\nu \mu }}\tilde{\alpha} _{E_{\nu \mu }}^{K,i}(\Phi^{n+1})e^{q^i \lambda^K(q_{\mu })  \Phi ^{n+1}_K} , {\nu >\mu }, \\
 & -\omega^K_{E_{\mu \nu }}\tilde{\alpha} _{E_{\mu \nu }}^{K,i}(\Phi^{n+1})e^{q^i \lambda^K(q_{\mu })  \Phi ^{n+1}_K} , {\nu <\mu }, \\
 &\sum\limits_{ k> \mu}\omega^K_{E_{\mu  k}}\tilde{\alpha} _{E_{\mu k}}^{K,i}(\Phi^{n+1})e^{q^i \lambda^K(q_{\mu })  \Phi ^{n+1}_K}
 \\&+ \sum\limits_{ k<\mu} \omega^K_{E_{k\mu}}\tilde{\alpha} _{E_{k\nu_2}}^{K,i}(\Phi^{n+1})e^{q^i \lambda^K(q_{\mu })  \Phi ^{n+1}_K} ,  {\nu =k},
\end{array}\right. ~\nu ,\mu=1,2,3,4,~i=1,2,
\end{equation}}
where~$\omega^K_{E_{\nu \mu }}$ is defined by~\eqref{omega}, $q_1,q_2,q_3,q_4$ are four vertices of element~$K$ (see Fig. \ref{figK}), and
\begin{equation}\label{[1](3.10)}
\tilde{\alpha}^{K,i}_{E_{\nu \mu } }(\Phi^{n+1} )=\left[\frac{1}{|\tau_{E }|}\int_{E } e^{q^i \lambda^K  \Phi ^{n+1}_K}\mathrm{d}s\right]^{-1}.
\end{equation}

\begin{figure}[H]
 	\centering
 	\includegraphics[width=4cm,height=4cm]{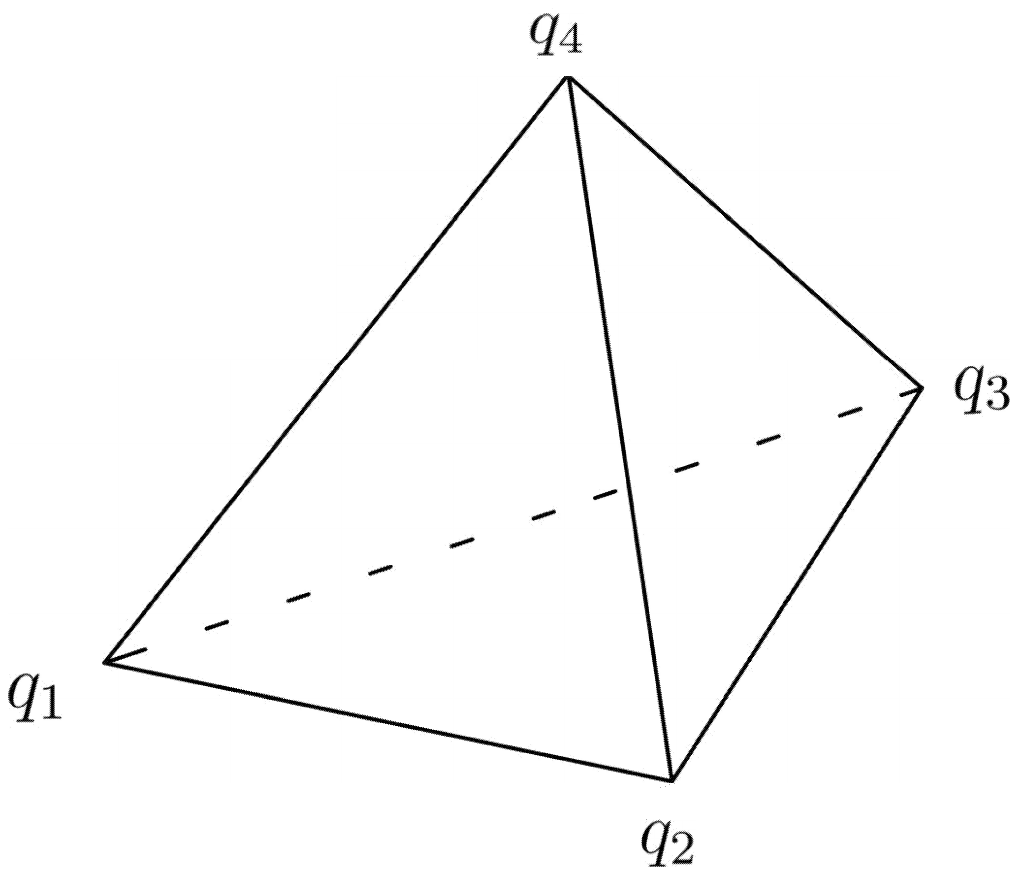}
 	\caption{Tetrahedral element $K$ }
 	\label{figK}
 \end{figure}

 From~\eqref{FEM-non}, \eqref{SUPG-non} and~\eqref{EAFE-non}, we have
\begin{equation}\label{tildebar}
  \Phi^{ n+1,i}(P^{ n+1,i}):=G_A^{n+1}-\bar BP^{ n+1,i},~\bar B:=A_L^{-1}\bar M,~G^{n+1}_A:=A_L^{-1}G^{n+1 }_\Phi,~i=1,2,3.
\end{equation}
Hence, in view of~\eqref{FEM-non}, \eqref{SUPG-non}, and~\eqref{EAFE-non}, it follows that the nonlinear subsystem with respect of $P^{i,n+1}$ on time layer~$t=t^{n+1}$ is
\begin{eqnarray}\label{EAFE-non-2}
A^i(P^{n+1})P^{n+1,i}= F^{n,i},~i=1,2,3,
\end{eqnarray}
where
\begin{eqnarray}\label{EAFE-non-2-AP}
A^i (P )= {\bf M} + \tau\tilde A^i(P ), ~\tilde  A^i(P ):= \bar A^i(\Phi(P )),~\tilde A^i(P )=diag(\tilde A^{1,i}(P ),\tilde A^{2,i}(P )) .
\end{eqnarray}

It is easy to know that it suffices to prove the existence of the solution of nonlinear subsystem~\eqref{EAFE-non-2} in order to show the existence of the solutions of the three nonlinear fully-discrete scheme ~\eqref{FEM-non}, \eqref{SUPG-non} and~\eqref{EAFE-non}. The following assumption is needed to show the existence of the solution to \eqref{EAFE-non-2}.
\begin{assu}\label{assum-M}
The stiffness matrix $A^i (P^{n+1}),~i=1,2,3$ defined by~\eqref{EAFE-non-2-AP} is a column M-matrix.
\end{assu}

The following Lemma \ref{corollary-2} proved that Assumption
\ref{assum-M} holds under some certain conditions for EAFE scheme. For similicity, the superscript $3$ representing EAFE scheme is omitted in the proof, and denote
the element stiffness matrix on element $K$ of~$\tilde A^{i}(P^{n+1})$ by
\begin{eqnarray}\label{tildeA(P)}
\tilde A^{i,K}(P^{n+1})=(\tilde  a_{\nu \mu }^{i,K}(P^{n+1}))_{4 \times 4},~i=1,2.
\end{eqnarray}

\begin{lemma} \label{corollary-2} If for any element~$K\in \mathcal {T}_h$,  $ \omega^K_{E _{\nu \mu }}$ defined by \eqref{omega} satisfies
\begin{equation}\label{cytj}
 \omega_{E _{\nu \mu }}^K>0,~\nu <\mu  ,~\nu ,\mu =1,2,3,4,
\end{equation}
then~$\tilde  A(P^{n+1})$ and~$A(P^{n+1})$ are all column M-matices.
 \end{lemma}
\begin{proof}
It can be verified directly from ~\eqref{EAFE-non-2-AP} and~\eqref{aij} that the column sum of the elements of~$\tilde A^{i,K}$ is zero and the off-diagonal element
\begin{equation*}\label{aijK}
\tilde  a_{\nu \mu }^{i,K}(P^{n+1})=\left\{\begin{array}{ll}
 -\omega ^K_{E_ {\nu \mu }}\tilde{\alpha} _{E_{\nu \mu }}^{K,i}(\Phi^{n+1})e^{q^i \lambda^K(q_{\mu })  \Phi^{n+1} _K} , &  {\nu >\mu } \\
  -\omega^K_{E_ {\mu  \nu }}\tilde{\alpha} _{E_{\mu \nu }}^{K,i}(\Phi^{n+1})e^{q^i \lambda^K(q_{\mu })  \Phi ^{n+1}_K} , &  {\nu <\mu }
\end{array}\right.,\nu ,\mu =1,2,3,4,~i=1,2.
\end{equation*}
Hence, from~\eqref{[1](3.10)} we know the sufficent and necessary conditions for all diagonal elements of~$\tilde A^{i,K}$ to be less than zero are~\eqref{cytj} holds. From
~\eqref{cytj}, it is easy to know the diagonal element of ~$\tilde A^{i,K}$ is
\begin{align*}\label{aiiK}
\tilde  a_{\mu \mu}^{i,K}(P^{n+1})=&\sum\limits_{ k> \mu}\omega^K_{E_{\mu k}}\tilde{\alpha} _{E_{\mu k}}^{K,i}(\Phi^{n+1})e^{q^i \lambda^K(q_{\mu})  \Phi ^{n+1}_K}
\\&+ \sum\limits_{ k< \mu} \omega^K_{E_{k\mu}}\tilde{\alpha} _{E_{k\mu}}^{K,i}(\Phi^{n+1})e^{q^i \lambda^K(q_{\mu})  \Phi ^{n+1}_K}>0,~\mu=1,2,3,4,
\end{align*}
which implies for any given element~$K$,  the element stiffness matrix~$\tilde A^{i,K}(P^{n+1}) $ is an
~L-matrix and the column sum of its elements is zero. Obviously, the global stiffness matrix integrated from
~$\tilde A^{i,K}(P^{n+1})$ is also an L-matrix and the column sum of its elements is zero. Note that $\tilde A^{i}(P^{n+1})$  is the global stiffness matrix obtained after the Dirichlet boundary treatment of $\tilde A_b^{i}(P^{n+1})$, that is, there is at least one column of strictly diagonally dominant matrix, which implies ~$\tilde A ^{i}(P^{n+1})$ is a column M-matrix. Hence, from \eqref{EAFE-non-2-AP} we have~$\tilde A(P^{n+1})$  is a column M-matrix, which combining with ~\eqref{EAFE-non-2-AP} and~\eqref{C4-M-def} yields~$A (P^{n+1})$ is a column M-matrix too.

\end{proof}

\begin{remark}
The condition~\eqref{cytj} can be replaced by the requirement of mesh quality, e.g. each element in the triangular partition is acute triangle in the two-dimensional case
\textsuperscript{\cite{xu1999monotone}}.

\end{remark}

It is shown in~\cite{prohl2009convergent} that if $h$ and~$\tau$ are small enough, and  $\mathcal T_h$ is a strong acute angle partition,  Assumption \ref{assum-M} holds for the standard FE scheme. Note that
from~\eqref{SUPG-AF}, we have $A^2(P^{n+1})=A^1(P^{n+1})+A^{stable}(\Phi^{n+1})$, and $A^{stable}(\Phi^{n+1})$ is a column M-matrix when the diffusion coefficient satisfies certain conditions. Hence if $\mathcal T_h$ is a strong acute angle partition,  $h$ and $\tau$ are small enough and the diffusion coefficient satisfies certain conditions{\color{blue}(e.g. ~$\nabla \phi=(a,-a,a)~or ~(-a,a,-a) ,~a\in R$)}, then Lemma \ref{assum-M} holds for ~SUPG scheme.

 	
Next, we will show the existence of the solution to the nonlinear subsystem \eqref{EAFE-non-2} under Assumption ~\ref{assum-M}. We only prove it for the EAFE scheme~\eqref{EAFE-non} as an example.  Setting the time layer $t=t_{J+1}$, $J\in\{0,\ldots,N-1\}$,
then ~\eqref{EAFE-non-2} can be written as
\begin{equation}\label{UphiJJ}
 P=\phi^{J+1}(P),~\phi^{J+1}(P):= (A (P))^{-1}F^J,~P:=P^{J+1}.
\end{equation}

Note that $P^0=(p^{0,1}_1,\ldots,p^{0,1}_{n_h},p^{0,2}_1,\ldots,p^{0,2}_{n_h})^T$ is the vector composed of the initial ion concentration values at nodes. Assume there exists a positive constant $C_{p^0}$, such that
\begin{equation}\label{pn}
\min_{k\in\{1,\ldots, n_h\},i= 1,2 }  p^{0,i}_k \ge C_{p^0}>0.
\end{equation}
Note that the assumption (\ref{pn}) is presented based on the physcical background of the solution. It is pointed our later in Remark \ref{init-assume} that this assumption can be removable from mathematical view.


Then we have the following property and lemma.
\begin{property}\label{pro1}
For any element  $K\in \mathcal T_h$, the general element $\tilde  a_{\nu \mu }^{i,K}(P^{n+1})$ of the element stiffness matrix defined by  \eqref{tildeA(P)} is smooth enough with respect to~$P^{n+1}$.

In fact, from \eqref{tildebar} we know~$\Phi^{n+1}_K(P^{n+1})$ is a linear polynomial function with respect to $P^{n+1}$, which follows that $e^{ \lambda^K(q_{\mu })  {\Phi}^{n+1}_K} $ is smooth enough with repect to~$P^{n+1}$. Hence in view of \eqref{aij}, ~\eqref{[1](3.10)} and the fact that ~$\omega$ is a constant independent of $P^{n+1}$, we have ~$\bar a_{\nu \mu }^{i,K}(\Phi^{n+1}(P^{n+1})) $ is smooth enough with respect to $P^{n+1}$.
\end{property}

Using Property~\ref{pro1}, noting that~$\tilde A^i(P^{n+1})$ is assembled from ~$\tilde A^{i,K}(P^{n+1}),~\forall K\in \mathcal T_h$, and from~\eqref{EAFE-non-2-AP}, it yields any element of the stiffness matrices~$\tilde A^i(P^{n+1}),i=1,2$ or~$\tilde A(P^{n+1})$ is smooth enough with respect to $P^{n+1}$.

\begin{lemma}\label{lemma1}
Assume $f$ and $F^i$ in (\ref{weak1con})-(\ref{weak2con}) satisfies
\begin{equation}\label{Fi}
f  , ~F^i  \in {L^{\infty}(0,T;L^{\infty}(\Omega))} ,~i=1,2. 
\end{equation}
If there exists a positive constants~$C_{p^{J}}$ satisfying
\begin{equation}\label{pn-J}
\min_{k\in\{1,\ldots, n_h\},i= 1,2 }  p^{{J},i}_k \ge C_{p^{J}}>0,
\end{equation}
then there exists a positive constant $\tau_{*,J}$ indepedent of $P^J$, such that
\begin{equation}\label{FJge0}
F^J> 0,~\mbox{if}~\tau< \tau_{*,J}.
\end{equation}
	
\end{lemma}

\begin{proof}
	Using the definitions of $F^J$ and ${\bf M}$ in ~\eqref{def-F}  and \eqref{C4-M-def-2}, respectively, we know the $k$th element of $F^{i,J}$ satisfies
	\begin{equation} \label{fgge0}
f_k^{i,J}=\tau g_k^{i,J+1}+\frac{|\Omega_k|}{4}p^{i,J}_k,i=1,2,
\end{equation}
where $g_k^{i,J+1}$ and $p^{i,J}_k $ are the $k$th element of $G^{i,J+1}$ and $P^{i,J}$, respectively.
	
	If $G^{J+1}=\bf 0$, then \eqref{FJge0} holds for any  $\tau$ by \eqref{fgge0} and \eqref{pn-J}. Next we consider the case $G^{J+1}\ne \bf 0$. From the definition of $G^{J+1}$ in \eqref{def-PG} and the assumption \eqref{Fi}, it is obvious  $\|G^{J+1}\|_{L^{\infty}(\Omega)}$ is bounded. Hence we can set
	\begin{equation}\label{def-tau*}
	\tau_{*,J}= \frac{C_{p^J}\min\limits_{k} |\Omega_k|}{4\|G^{J+1}\|_{L^{\infty}(\Omega)}}>0.
	\end{equation}
	Using \eqref{fgge0}, \eqref{def-tau*} and \eqref{pn-J}, for $\tau<\tau_{*,J}$, there holds
	\begin{eqnarray*}\notag
f_k^{i,J}&=&\tau g_k^{i,J+1 }+\frac{|\Omega_k|}{4}p^{i,J }_k
>-\frac{C_{p^J}\min _{k} |\Omega_k|}{4\|G^{J+1}\|_{L^{\infty}(\Omega)}}|g_k^{i,J+1 }|+\frac{|\Omega_k|}{4}p^{i,J }_k
\\&\ge&-\frac{C_{p^J}\min _{k} |\Omega_k|}{4  } +\frac{|\Omega_k|}{4}p^{i,J }
\ge\frac{  C_{p^J}}{4  }(- {\min _{k} |\Omega_k|} + {|\Omega_k|} )\ge0.\label{gle0}
\end{eqnarray*}
	Thus we get \eqref{FJge0}. This completes the proof of this lemma.
\end{proof}

\begin{lemma}\label{lemma2}
	Let $E^T=(1,\ldots,1)^T\in R^{2n_h}$ and assume \eqref{cytj} and \eqref{FJge0} hold.  
	Define the positive constants
	\begin{equation}\label{def-CJ}
	C_J= EF^J,~C^{J+1}_k=\frac{4C_J}{|\Omega_k|},
	\end{equation}
	where $F^J$ is defined by \eqref{FJge0},
	and compact convex set
	\begin{equation}\label{set-JJ}
	\mathcal{C}^{J+1}=\{ P~| P\in \mathbb{R}^{2{n_h}}, p_k^i\in[0,C_k^{J+1}],~k=1,2,\ldots,{n_h},i=1,2\}.
	\end{equation}
	Then $\phi^{J+1}$ in \eqref{UphiJJ} is a continuous mapping from $\mathcal C^{J+1}$ to $\mathcal C^{J+1}$, if $\tau<\tau_{*,J}$, where $\tau_{*,J}$ is given in \eqref{FJge0}.
\end{lemma}

\begin{proof}
	Since $A(P)$ is a column M-matrix, we get $A(P)$ is nonsingular. Then from Property~\ref{pro1}, we have $A^{-1}(P)$ is smooth enough. From \eqref{UphiJJ}, it follow that $\phi^{J+1}$ is continuous in~$\mathcal C^{J+1}$. Hence it suffices to prove $\phi^{J+1}$ is a mapping from $\mathcal C^{J+1}\rightarrow \mathcal C^{J+1}$, that is
	\begin{equation*}\label{eq}
	\bar P=\phi^{J+1}(P )\in \mathcal C^{J+1}, ~\forall P \in \mathcal C^{J+1}.
	\end{equation*}
	From the definition of $ \mathcal C^{J+1}$ in~\eqref{set-JJ}, we only need to show the general element $\bar p^i_{k}$ of $\bar P$ satisfies
	\begin{equation}\label{Ci-J}
	\bar p^i_{k} \in [0,C^{J+1}_k], ~ k=1,\ldots, {n_h},i=1,2.
	\end{equation}
From~\eqref{UphiJJ}, we know~$\bar P  =\phi^{J+1}(P )$ is equivalent to the following equation
	\begin{equation}\label{eqJJ}
	A (P)\bar P =F^J.
	\end{equation}
	Note that since~$A^T(P)$ is an M-matrix,  the nonsingular matrix $B=(b_{j_1j_2})_{2n_h\times2n_h}:=(A^T(P))^{-1}$ is a nonnegative matrix, i.e. there is at least one positive element on each row of $B^T$. Then using  \eqref{eqJJ} and \eqref{FJge0}, we have the general element $p^i_{k}$ of $\bar P$ satisfies
	\begin{equation}\label{PJge0}
	\bar p^i_{k}= ( B^TF^J)_k
	>0,~ k=1,\ldots, {n_h},i=1,2.
	\end{equation}
	Hence to show~\eqref{Ci-J}, it only need to prove
	\begin{equation}\label{Ci-J1}
	\bar p^i_{k} \le C_k^{J+1}, ~ k=1,\ldots, {n_h}, ~i=1,2.
	\end{equation}
    From ~\eqref{def-CJ}, \eqref{eqJJ} and ~\eqref{EAFE-non-2-AP}, we have
	\begin{equation}  \label{C4-eq1.39-J}
	C_{J} =E F^J=E  A (P ) \bar P =E  ({\bf M}+\tau \tilde  A(P )) \bar P .
	\end{equation}

	From Corollary~\ref{corollary-2} and \eqref{PJge0}, it follows that
	$E\tilde  A(P) \bar P> 0$, which combining with
	~\eqref{C4-eq1.39-J}, \eqref{C4-M-def-2} and~\eqref{def-Pn} yields
	\begin{eqnarray*}
		C_J&=&{\color{black}E  ({\bf M}+\tau \tilde  A(P )) \bar P }
		> E {\bf M}  \bar P  = \sum_{j=1}^{n_h} \frac{|\Omega_j|}{4}( \bar p^1_{ j}+\bar  p^2_{ j}).
	\end{eqnarray*}
	From \eqref{PJge0} and using the fact  $C^{J+1}_k=\frac{4C_J}{|\Omega_k|}>0$, we get
	\begin{eqnarray}\label{CJ}
	C_J > \frac{|\Omega_k|}{4}  \bar p^i_{ k}=\frac{C_J}{C^{J+1}_k} \bar p^i_{ k}, ~\forall k\in\{1,\ldots, n_h\},i\in\{1,2\},
	\end{eqnarray}
	which implies ~\eqref{Ci-J1}. We complete the proof of this lemma.
\end{proof}

From Lemmas \ref{lemma1} and \ref{lemma2}, and Brouwer fix point theorem (see~\cite{park1999ninety}), we have the following existence theorem.

\begin{theorem}  \label{th-1}
	Assume  \eqref{cytj}, \eqref{pn}  and \eqref{Fi} hold. Then the solution of nonlinear system \eqref{UphiJJ} exists on any given time step $t=t^n,n\in\{1,\ldots,N\}$.
\end{theorem}

\begin{proof}
	We prove Theorem \ref{th-1} by the mathematical induction. First we show the solution of the nonlinear system  \eqref{UphiJJ} exists on the time step $t=t^1$.
	Since \eqref{pn} is the case $J=0$ in \eqref{pn-J}, combining \eqref{Fi}, from Lemma \ref{lemma1} we obtain \eqref{FJge0}.  Using \eqref{cytj}, \eqref{FJge0} and Lemma \ref{lemma2}, we have $\phi^1$ is a continuous mapping from~$\mathcal C^1\rightarrow \mathcal C^1$, which combining with Brouwer fixed point theorem yields the solution $P^1$ of ~\eqref{UphiJJ} exists,  i.e.
	\begin{equation*}\label{eq1J}
	A (P^{1})  P^{1} =F^{0}.
	\end{equation*}
	Since $F^0>0$ and $(A(P^{1}))^T$ is an M-matrix, using the similar arguments in the proof of \eqref{PJge0}, we get the element $p^{i,1}_{k}$ of $P^{1}$ satisfies
	\begin{equation}\label{P1Jge0-1}
	p^{i,1}_{k}= ( B_{1}^TF^0)_k
	>0,~B_{1}=((A(P^{1}))^T)^{-1},~ k=1,\ldots, {n_h},i=1,2.
	\end{equation}
	Hence there exists a positive constant $C_{p^{1}}$ such that
	\begin{equation*}\label{p1}
\min_{k\in\{1,\ldots, n_h\},i= 1,2 }  p^{i,{1} }_k \ge C_{p^{1}}>0,
\end{equation*}
	where $ p^{i,{1} }_k$ is the element of $P^{1}$.
	
	Assuming the solution of \eqref{UphiJJ} exists on $t=t^m(m\ge 1)$ and
	\begin{equation}\label{pJ}
\min_{k\in\{1,\ldots, n_h\},i= 1,2 }  p^{i,{m}}_k \ge C_{p^{m}}>0,
\end{equation}
where~$C_{p^{m}}$ is a constant independent of $P^{m}$, we shall show the solution also exists on $t=t^{m+1}$ and  \eqref{pJ} holds.

From \eqref{Fi}, \eqref{pJ} and setting $J=m$, then we get \eqref{FJge0} by Lemma \ref{lemma1}. Further, using \eqref{cytj}, \eqref{FJge0} and Lemma \ref{lemma2}, we have $\phi^{m+1}$ is a continuous mapping from~$\mathcal C^{m+1}\rightarrow \mathcal C^{m+1}$. Thus, using Brouwer fixed point theorem, it follows that the solution of \eqref{UphiJJ} exists on $t=t^{m+1}$, that is
	\begin{equation*}
	A (P^{m+1})  P^{m+1} =F^{m}.
	\end{equation*}
	Then similar as the proof of \eqref{P1Jge0-1} for the case $t=t^1$, we get any element $p^{i,m+1}_{k}$ of $P^{m+1}$ satisfies
	\begin{equation*}
 p^{i,m+1 }_{k}= ( B_{m+1}^TF^0)_k
>0,~B_{m+1}=((A(P^{m+1}))^T)^{-1},~ k=1,\ldots, {n_h},i=1,2,
\end{equation*}
	Hence there is a positive constant $C_{p^{m+1}}$ such that
	\begin{equation*}\label{p1}
\min_{k\in\{1,\ldots, n_h\},i= 1,2 }  p^{i,{m+1} }_k \ge C_{p^{1}}>0,
\end{equation*}
	where $ p^{i,{m+1} }_k$ is the element of  $P^{m+1}$.
	This completes the proof of the lemma.
\end{proof}

\begin{remark}\label{init-assume}
	
	{\color{black}
Theorem \ref{th-1} shows that the existence of the solution to nonlinear system \eqref{EAFE-non} is based on \eqref{pn}, which can be removed essentially.}
\end{remark}	
	Next,  we provide a detailed explanation on how to obtain the result of Theorem \ref{th-1}  without using assumption \eqref{EAFE-non}.
Assume the initial function
	$p^{i}_0\in C(\bar \Omega)$.
	Let
	\begin{equation}\label{tildep}
	\tilde p^{i}({\bf x},t) = p^{i} ({\bf x},t)+C_0,
	\end{equation}
	where the positive constant ~$ C_0:=\max\limits_{{\bf x}\in\bar\Omega} \{|p^{i,0}_{\bf x}|\}+C_{p^0}$.
	Due to the difference of only one constant between~$p^{i}$ and~$\tilde p^{i}$, it is only need to prove the existence of $\tilde p^{i}$ to show the existence of ~$p^{i}$.
	
	It is easy to know $\phi$ and $\tilde{p}^i$ satisfy the following equations
	\begin{equation}\label{pnp-equation-1}
	\left\{\begin{array}{lr}
	-\Delta  \phi-\sum\limits^{2}_{i=1}q^{i}\tilde p^{i}=f ,  &\text { in } \Omega,~\text{for}~t\in(0, T],  \\
	{\partial_t  \tilde p^{i} }-\nabla\cdot(\nabla \tilde p^{i}  + q^{i} \tilde p^{i} \nabla  \phi   ) +q^{i}C_0\sum\limits^{2}_{j=1}q^{j}\tilde p^{j }  =\tilde F^i, & \text { in } \Omega,~\text{for}~t\in(0, T], i=1,2,\vspace{1mm}
	\end{array}\right.
	\end{equation}
	where~$\tilde F^i:=F^i- q^{i}C_0 f$, with the initial-value conditions
	\begin{equation}\label{boundary-1-0}
	\left\{
	\begin{array}{lr}
	\phi=0,\text{on}\;\partial\Omega,& \text{for}~t\in(0, T],\vspace{0.5mm}\\
	\tilde p^i=C_0,\;\text{on}\;\partial\Omega,&\text{for}~t\in(0, T],\\
	\tilde p^{i}({\bf x},0)=p^{i}({\bf x},0)+ C_0 ,&\text{for}~{\bf x} \in \Omega.
	\end{array}
	\right.
	\end{equation}
	We know that  $\tilde p^{i}({\bf x},t)$ satisfies
	\begin{equation}\label{pn-1}
	\min_{k\in\{1,\ldots, n_h\},i= 1,2 }  \tilde{p}^{0,i}_k \ge C_{p^0}>0,
	\end{equation}
	where  $\tilde{P}^0=(\tilde{p}^{0,1}_1,\ldots,\tilde{p}^{0,1}_{n_h},\tilde{p}^{0,2}_1,\ldots,\tilde{p}^{0,2}_{n_h})^T$ is the vector composed of the initial ion concentration values at nodes and  $C_{p^0}$ is a certain positive constant.
	
	Let the vector of degree of freedoms of $\tilde p^{i}$ be
	\begin{eqnarray}
	\tilde P=
	\left(
	\begin{array}{c}
	\tilde P_{b} \\
	\tilde P_{I}\\
	\end{array}
	\right),
	\end{eqnarray}
	where $\tilde P_{b} $  and $\tilde P_{I}$ are vectors of degrees of freedoms corresponding to the set of boundary nodes and the set of inner nodes.
	
	Similar as the deduction of the EAFE scheme \eqref{EAFE-non} from \eqref{pnp equation}, using \eqref{pnp-equation-1}-\eqref{boundary-1-0} and noting the boudary condition $\tilde p^i=C_0$ in~\eqref{boundary-1-0} can be written as $\alpha\tilde p^i=\alpha C_0$ ($\alpha$ is any constant but not zero),  it is easy to know that the discrete system corresponding to the backward Euler-EAFE scheme of \eqref{pnp-equation-1}-\eqref{boundary-1-0} is as follows
	\begin{eqnarray}\label{AUF}
\left(
  \begin{array}{cc}
    A_L& \bar M \\
   {\bf 0 }   &   \hat A(\tilde P)  \\
  \end{array}
\right)
\left(
  \begin{array}{c}
      \Phi\\ \tilde P
  \end{array}
\right)=\left(
  \begin{array}{c}
    G_\Phi^{n+1}\\
   \tilde F^n
  \end{array}
\right),
\end{eqnarray}
	Here $A_L$, $\bar M$, $\Phi$ and $G_\Phi^{n+1}$ are defined in \eqref{C4-AL-def}, \eqref{C4-M-def} and  \eqref{def-PhiG3}, respectively,  $2(n_{h}+n_{bd})\times 2(n_{h}+n_{bd})$ matrix and $2(n_{h}+n_{bd})$ dimensional vector
	\begin{equation}\label{tidleAF}
	\hat A(\tilde P)= \left(
	\begin{array}{cc}
	\alpha I_b  & {\bf 0} \\
	A_{b}(\tilde P)   &   {A}_I(\tilde P) \\
	\end{array}
	\right),~ \tilde F^n=
	\left(
	\begin{array}{c}
	\alpha\tilde P_e \\
	\tilde F_I^n \\
	\end{array}
	\right),
	\end{equation}
	where $n_{bd}$ is the number of nodes on the boundary, $I_b$ is $2n_{bd}\times 2n_{bd}$ identical matrix,
	$2n_{h}\times 2n_h$ matrix $A_{b}(\tilde P)$ is nonpositive matrix assembled from ~$\tau\tilde A^{i,K}(P)$, $2n_{h}\times 2n_h$ matrix~$ {A}_I(\tilde P)=A(P)+C_0\tau\check{M}$, ${\check M} =\begin{pmatrix}
	M&-M
	\\
	-M &M
	\end{pmatrix}$,   $\tilde P_e$ is the vector composed by the values of the exact solution on boundary nodes, and the right-hand vector
	$$\tilde F_I^n  = \tau\tilde G^{n+1}+ {\bf M} \tilde P^n,
	~\tilde G^{n+1}=((\tilde F^{1,n+1},\psi_1),\ldots,(\tilde F^{1,n+1},\psi_{n_h}),(\tilde F^{2,n+1},\psi_1),\ldots,(\tilde F^{2,n+1},\psi_{n_h}))^T.$$

	From Lemma~\ref{corollary-2} and~$ {A}_I(\tilde P)=A(P)+C_0\tau\check{M}$, we have~${A}_I(\tilde P)$ is a column M-matrix, which combining with the definition of~$\hat A(\tilde P)$ given in~\eqref{tidleAF} yields that if choose
	\begin{equation}\label{def-alpha}
	\alpha\ge \|A_{b}(\tilde  P)\|_1,
	\end{equation}
	then~$\hat A(\tilde  P)$ is a column M-matrix.

In summary,  the result of Theorem \ref{th-1} can be obtained under the assumptions of \eqref{cytj} and \eqref{Fi} and without using the condition (\ref{pn}).

The existence of the solution to FE or SUPG schemes can been obtained by following the similar arguments in the proof of the existence of the solution to the EAFE scheme.

Note that Gummel iteration is a commonly used iteration to solve nonlinear scheme (\ref{fulldis2-EAFE})-(\ref{fulldis-EAFE}) (cf. e.g. \cite{zhang2022class, SUPG-IP, zheng2011second}). In next section, we introduce the Gummel algorithm combining with EAFE scheme for PNP equaitons, and present the contraction and convergence theory for the solution of the algorithm.
\vspace{3mm}
\subsection{The convergence analysis for the Gummel iteration}
In this section, we will present the convergence analysis for the Gummel iteration.
 For $t=t^{n+1},~n\geq 0$, the Gummel iteration for the EAFE scheme (\ref{fulldis2-EAFE})-(\ref{fulldis-EAFE}) are as follows.

 \begin{algorithm}[!h]
 	\caption{Gummel iteration}
 	\label{gummalgo}
 	Step 1. Give the initial value $(p_h^{1,n},p_h^{2,n},\phi_h^n)\in [V_h]^3$, let $(p_h^{1,n+1,0}, p_h^{2,n+1,0},\phi_h^{n+1,0})$ $=(p_h^{1,n},p_h^{2,n},\phi_h^n)$ as $l=0$.\\
 	Step 2. For $l\geq 0$, compute $(p_h^{1,n+1,l+1}, p_h^{2,n+1,l+1},\phi_h^{n+1,l+1})\in [V_h]^3$, such that for any $v_{h}$ and $ w_h\in V_h$,
 	\begin{align}\label{gumm2}
 	&(\nabla\phi_h^{n+1,l+1},\nabla w_h)-\sum\limits^{2}_{i=1}q^{i}(p^{i,n+1,l}_h,w_h)=(f^{3,n+1},w_h),
 	\\ \label{gumm1}
 	&\frac 1 {\tau}(p_h^{i,n+1,l+1},v_h)+ \sum\limits_{K\in \mathcal{T}_{h}}[\sum\limits_{E\subset K}\omega_E^K\tilde{\alpha}^{K,i}_{E}(\phi_h^{n+1,l+1})\delta_E(e^{q^i\phi_h^{n+1,l+1}}p_h^{i,n+1,l+1})\delta_Ev_{h}]
 	\nonumber\\
 	&~~~~~~=(F^{i,n+1},v_h)+\frac 1 {\tau}(p^{i,n}_h,v_h) ,~i=1,2 .
 	\end{align}
 	
 	Step 3. For a given tolerance $\epsilon$, stop the iteration if
 	\begin{align}\label{tol}
 	\|p_h^{1,n+1,l+1}-p_h^{1,n+1,l}\|+\|p_h^{2,n+1,l+1}-p_h^{2,n+1,l}\|+
 	\|\phi_h^{n+1,l+1}-\phi_h^{n+1,l}\|\leq\epsilon,
 	\end{align}
 	and set $(p_h^{1,n+1},p_h^{2,n+1},\phi^{n+1}_h)=(p_h^{1,n+1,l+1},
 	p_h^{2,n+1,l+1},\phi_h^{n+1,l+1})$. Otherwise set $l\leftarrow l+1$ and goto Step 2 to continue the iteration.
 \end{algorithm}

To show the convergence of the solution to Algorithm 3.1, first we present the algebraic formulation of Algorithm \ref{gummalgo}. Let the linear finite element functions of \eqref{gumm2} and \eqref{gumm1} are
$$\phi_h^{n+1,l+1}=\Psi \Phi^{n+1,l+1}, ~ p_h^{i,n+1,\mu}=\Psi P^{i,n+1,\mu},~ i=1,2,~\mu=l,l+1,$$
respectively, where the basis function vector $\Psi$  is given by \eqref{Psi}. The degrees of freedom vector of electrostatic potential and concentration are defined as follows
\begin{equation}\label{def-P-l+1}
\Phi^{n+1,l+1}=(\Phi^{n+1,l+1}_1,\ldots,\Phi^{n+1,l+1}_{n_h})^T,
~
P^{i,n+1,\mu}=(p^{i,n+1,\mu}_1,\ldots,p^{i,n+1,\mu}_{n_h})^T.
\end{equation}

Setting $w_h$ and $v_h$ in ~\eqref{gumm2}-\eqref{gumm1} to be the basis functions, using \eqref{def-P-l+1} and denoting by
 \begin{eqnarray}\label{EAFE-non-dis-1}
F_\Phi^{n+1}(P):=G_\Phi^{n+1}- \bar MP,
\end{eqnarray}
then the linear algebraic equations corresponding to ~\eqref{gumm2}-\eqref{gumm1} can be written as
	\begin{eqnarray}\label{EAFE-non-dis}
	\left\{
	\begin{array}{ll}
	A_L  \Phi^{n+1,l+1}= F_\Phi^{n+1}(P^{n+1,l}) , \\
	({\bf M} + \tau \bar A(\Phi^{n+1,l+1}))P^{n+1,l+1}= F^n ,
	\end{array}
	\right.~P^{n+1,\mu}=
	\begin{pmatrix}
	P^{1,n+1,\mu}\\
	P^{2,n+1,\mu}
	\end{pmatrix}
	,~\mu=l,l+1,
	\end{eqnarray}
where matrices $A_L$, $\bar A(\cdot)$, $\bf M$ and  $\bar M$ are defined in \eqref{C4-AL-def}  and \eqref{C4-M-def}, respectively, vectors $G_\Phi^{n+1} $ and $F^n$ are given by \eqref{def-PhiG3}  and \eqref{def-F}, respectively.

Similar as the deductions of~\eqref{tildebar} and \eqref{UphiJJ}, using \eqref{EAFE-non-dis} and \eqref{EAFE-non-2-AP}, we have
\begin{equation}\label{tildebar-dis}
\Phi^{n+1,l+1}(P^{n+1,l})=G_A^{n+1}-\bar BP^{n+1,l},~\bar B:= A_L^{-1}\bar M ,~G_A^{n+1}:=A_L^{-1}G_\Phi^{n+1}
\end{equation}
and
\begin{eqnarray}\label{EAFE-non-2-dis}
P^{n+1,l+1}=\phi^{n+1}(P^{n+1,l}),~\phi^{n+1}(P^{n+1,l}):=(A (P^{n+1,l}))^{-1}F^n.
\end{eqnarray}

Let ${\tilde C}_k^{0}:=\|P^0\|_\infty$ be a compact convex set and
\begin{equation}\label{set-JJ-3}
\mathcal{\tilde C}^{n}=
\{ P^l~| P^l\in \mathbb{R}^{2{n_h}}, p_k^{l,i}\in[0,{\tilde C}_k^{n}],~k=1,2,\ldots,{n_h},i=1,2\},~{\tilde C}_k^{n}=\max\{{\tilde C}_k^{n-1},{ C}_k^{n}\} ,~ n= 1,\ldots, N ,
\end{equation}
where the positive contant $C^{n}_k,~n=0,\ldots, N-1$ defined by \eqref{def-CJ}. From \eqref{set-JJ-3} and \eqref{def-CJ}, it is obvious to have
\begin{equation}\label{tildeC-C}
\mathcal { \tilde C}^{n-1}\subseteq\mathcal { \tilde C}^{n}~\mbox{and}~\mathcal { C}^{n}\subseteq\mathcal { \tilde C}^{n},~n= 1,\ldots, N .
\end{equation}

Next, in order to show the convergence of Algorithm \ref{gummalgo}, some lemmas needed are presented as follows.

\begin{lemma}\label{tildeC-P-l}
For any time step $t=t^{n},n\in\{1,\ldots,N\}$, if the Gummel iterative vector $P^{n,l}$ defined in \eqref{EAFE-non-dis} satisfies
	\begin{equation}\label{assume}
	P^{n,l}\in\mathcal { \tilde C}^{n},~\forall l \ge 0,
	\end{equation}
	then there exists a positive constant $C_L$ independent of $\{P^{n,\mu},\mu\ge 0\}$, such that the stiffness matrix $\tilde A(P)$ in \eqref{EAFE-non-2-AP} satisfies the following Lipschitz condition	\begin{eqnarray}\label{lp-l}
	\|\tilde A(P^{n,l})-\tilde A(P^{n,l-1})\|_{\infty}\le C_L \|P^{n,l}-P^{n,l-1}\|_{\infty},~\forall l \ge 1.
	\end{eqnarray}

	\begin{proof}
		In fact, since the diagonal block submatrix $\tilde A^i(P), i=1,2$ of $\tilde A(P)$ in  \eqref{EAFE-non-2-AP} is assembled by the element stiffness matrix~$\tilde A^{i,K}(P):=(\tilde  a_{\nu \mu }^{i,K}(P))_{4\times 4}$, to prove \eqref{lp-l}, it suffices to show there is a positive constant $C_L$ independent of  ~$\{P^{n,\mu},\mu\ge 0\}$, such that
		\begin{eqnarray}\label{lp-l1}
		|\tilde  a_{\nu \mu }^{i,K}(P^{n,l})-\tilde  a_{\nu \mu }^{i,K}(P^{n,l-1})| \le C_L \|P^{n,l}-P^{n,l-1}\|_{\infty},~\forall l \ge 1.
		\end{eqnarray}
		
		Next, we only give the discussion for the case $i=1$, and the case where i=2 is also similar. Let
		\begin{eqnarray}\label{tildeC-l}
		\eta_r ({\bf x}) =q^1\lambda^K({\bf x} ) \Phi_K^{n,r+1}(P^{n,r}),~r=l-1,l.
		\end{eqnarray}
		Using   \eqref{aij}  and \eqref{[1](3.10)}, for $\nu\ne \mu$ we have
		\begin{eqnarray}\label{a-a-l}
		|\tilde a_{{\nu \mu }}^{1,K} (P^{n,l} )-\tilde a_{{\nu \mu }}^{1,K}(P^{n,l-1}  )|
		\notag&\le &|\omega^K_{E_{\nu \mu }}|~ |\left[\frac{1}{|\tau_{E_{\nu \mu }}|}\int_{E_{\nu \mu }} e^{\eta_l (\bf x)}\mathrm{d}s\right]^{-1}e^{  \eta_l(q_{\mu })}-\left[\frac{1}{|\tau_{E_{\nu \mu } }|}\int_{E_{\nu \mu } } e^{  \eta_{l-1}({\bf x}) }\mathrm{d}s\right]^{-1}e^{  \eta_{l-1} (q_{\mu })}|
		\notag\\&&
		\notag\\&\le&|\omega^K_{E_{\nu \mu }}|~ |(e^{ \min\limits_{{\bf x}\in {E_{\nu \mu }}}  \eta_l ({\bf x}) })^{-1}e^{ \eta_l  (q_{\mu })}-(e^{ \max\limits_{{\bf x}\in {E_{\nu \mu }}}   \eta_{l-1} ({\bf x})   })^{-1}  e^{  \eta_{l-1}(q_{\mu })}|\label{delta-tilde-alpha-e}
		\notag\\&&
		\notag\\&\le & \frac{1}{3}   C_L^1   |e^{ \eta_l(q_{\mu })} -e^{  \eta_{l-1}(q_{\mu })}|
		\label{a-a-1}
		\end{eqnarray}
	 where $C_L^1=3\max\limits_{E_{\nu \mu }} \{|\omega^K_{E_{\nu \mu }}|\min\{(e^{ \min\limits_{{\bf x}\in {E_{\nu \mu }}} \eta_l ({\bf x} )})^{-1},(e^{ \max\limits_{{\bf x}\in {E_{\nu \mu }}} \eta_{l-1} ({\bf x})})^{-1} \}\}
	$.
		
		Next we show $C_L^1$ is independent of $P^{n,r},~r=l-1,l$. From \eqref{omega}, \eqref{tildeC-l} and the fact $\|\lambda^K({\bf x})\|_\infty\le 1$, it only requires  proving  there exists a positve constant $\hat C_L^1$ independent of $P^{n,r},r=l-1,l$, such that
		\begin{eqnarray}\label{tildeC1-l}
		\|\Phi^{n,r+1}\|_{\infty}\le  \hat  C_L^1,~r=l-1,l.
		\end{eqnarray}
		From  \eqref{def-PhiG3}, \eqref{C4-M-def},  \eqref{assume} and \eqref{set-JJ-3}, we get
		\begin{eqnarray}
		\|G_\Phi^{n}- \bar MP^{n,r}\|_{\infty}\le C_G+ \frac{\max\limits_k|\Omega_k|}{4}\| P^{n,r}\|_{\infty}
		\le C^F_0
		,~r=l-1,l,\label{FPHI-0-l}
		\end{eqnarray}
		where $C^F_0=C_G+ \frac{\max\limits_k|\Omega_k|}{4}\max\limits_k \{  { \tilde C}^{n+1}_{k}\}$ is a positive constant independent of $P^{n,r},r=l-1,l$.
		
		From \eqref{EAFE-non-dis}, we know $  \Phi^{n,r+1}= A_L^{-1}F_\Phi^{n}(P^{n,r})$. Then using  \eqref{EAFE-non-dis-1} and \eqref{FPHI-0-l}, and noting  $\|A_L^{-1}\|_{\infty}$ is only related to $h$, we obtain \eqref{tildeC1-l}.
		
	Next, we estimate \eqref{a-a-1} $ |e^{ \eta_l(q_{\mu })} -e^{  \eta_{l-1}(q_{\mu })}|$. Let $a  =  \min\{   \eta_1 (q_{\mu }),\eta_2  (q_{\mu })  \},~b=\max\{   \eta_1  (q_{\mu })  , \eta_2 (q_{\mu })\}$. From Lagrange mean value theorem, \eqref{tildeC-l} and \eqref{tildebar-dis},
	 it follows that
		\begin{eqnarray}\notag
		|e^{ \eta_l(q_{\mu })} -e^{  \eta_{l-1}(q_{\mu })}|
		&=&  { \color{black}{ | \eta_l (q_{\mu }) -\eta_{l-1} (q_{\mu })  | }\frac{|e^{ \eta_l  (q_{\mu })} -e^{  \eta_{l-1}(q_{\mu })}| }{ | \eta_l  (q_{\mu }) -\eta_{l-1}  (q_{\mu })  | }}
		\notag\\&=&   {| \eta_l  (q_{\mu }) -\eta_{l-1} (q_{\mu })  |}~|(e^\xi)'|
		\notag\\&\le & \| e^{   x}  \|_{L^{\infty}(a,b) }  ~{ |    \lambda^K(q_{\mu })(\Phi_K^{n,l +1}(P^{n,l })-\Phi_K^{n,l } (P^{n,l-1}))  | }
		\notag\\&\le & 4\| e^{  x}  \|_{L^{\infty}(a,b) } ~\| \lambda^K \|_\infty~{ \|   ( \bar B( P^{n,l} - P^{n,l-1}))^K \|_{\infty}} 
		\notag
		\notag\\&  \le &  C_L^2\|  P^{n,l} - P^{n,l-1} \|_{\infty}
		,\label{a-a-3}
		\end{eqnarray}
	where $C_L^2=4 \| e^x  \|_{L^{\infty}(a,b) }\|\bar B\|_\infty$. Noting that $\|\bar B\|_\infty$ is a positive constant only dependent of $h$, we have $C_L^2$ is independent of $P^{n,r},~r=l-1,l$.

		Inserting \eqref{a-a-3} into \eqref{a-a-1} and setting $C_L=C_L^1C_L^2$,  we have \eqref{lp-l1} holds when $\nu\ne\mu$. Then from \eqref{aij}, when $\nu=\mu$, it yields
		{\small\begin{eqnarray*}
				|\tilde a_{{\mu \mu }}^{1,K} (P^{n,l} )-\tilde a_{{\mu \mu }}^{1,K}(P^{n,l-1}  )|
				\le\sum\limits_{k \ne \mu} |\tilde a_{{k \mu }}^{1,K} (P^{n,l})-\tilde a_{{k \mu }}^{1,K}(P^{n,l-1}  )|
				\le     C_L \|P^{n,l}-P^{n,l-1}\|_{\infty},
		\end{eqnarray*}}
		thus we get \eqref{lp-l1}. This complete the proof of this lemma.
	\end{proof}
\end{lemma}

Similar as the deduction of Lemma \ref{lemma2} and Property~\ref{pro1}, we have $\phi^{n+1}$ and $\tilde A(P)$ are continuous on $\mathcal {\tilde C}^{n+1}$. Hence $\phi^{n+1}$ and  $\tilde A(P^{n+1,l})$ are bouned on the compact convex set $\mathcal {\tilde C}^{n+1}$. Denote by 
\begin{equation}\label{C-Phi-tildeA-l}
C_{\Phi}^{\tilde A} =\max \{\|\tilde A(P^{n+1,l})\|_{\infty},\| \phi^{n+1}(P^{n+1,l})\|_{\infty},l=0,1,\ldots,\}.
\end{equation}

\begin{lemma}\label{le-infty-l}
Under the assumption of Lemma \ref{tildeC-P-l}, and setting
	\begin{equation}\label{def-tauC-l}
	\tau_C=\frac{\min\limits_{k}|\Omega_k|}{8 C_L  C_{\Phi}^{\tilde A}},
	\end{equation}
where $|\Omega_k|$, $C_L$ and $ C_{\Phi}^{\tilde A} $ are constants defined in \eqref{C4-M-def}, \eqref{lp-l} and \eqref{C-Phi-tildeA-l}, respectively, then for $\tau<\tau_C$, the mapping $\phi^{n+1}$ given in \eqref{EAFE-non-2-dis} is contractive, i.e.
	\begin{equation}\label{contraction-l}
	\|\phi^{n+1}(P^{n+1,l})-\phi^{n+1}(P^{n+1,l-1})\|_{\infty}\le \alpha_{n+1}\| P^{n+1,l}-  P^{n+1,l-1} \|_{\infty},
	\end{equation}
	where~$ \alpha_{n+1}=\frac{8  C_L  C_{\Phi}^{\tilde A} }{\min\limits_{k}|\Omega_k|}{\tau}\in(0,1)$ are a positive constant independent of $P^{n+1,l-1}$ and $P^{n+1,l}$.
\end{lemma}

\begin{proof}
	Using~\eqref{EAFE-non-dis} and \eqref{EAFE-non-2-AP}, we obtain
	$$
	({\bf M}+\tau\tilde A(P^{n+1,l}))\phi^{n+1}(P^{n+1,l})=F^n=({\bf M}+\tau\tilde A(P^{n+1,l-1}))\phi^{n+1}(P^{n+1,l-1}),
	$$
	which becomes
{\small\begin{eqnarray}\notag
	{\bf M}(\phi^{n+1}(P^{n+1,l})-\phi^{n+1}(P^{n+1,l-1}))&+&\frac{\tau}{2} \big(\tilde A(P^{n+1,l})\phi^{n+1}(P^{n+1,l})-\tilde A(P^{n+1,l-1})\phi^{n+1}(P^{n+1,l-1})\big)
	\\&=&\frac{\tau}{2}\big(\tilde A(P^{n+1,l-1})\phi^{n+1}(P^{n+1,l-1})-   \tilde A(P^{n+1,l})\phi^{n+1}(P^{n+1,l})\big).
	\label{3-1-3'c11-l}
	\end{eqnarray}}

	From \eqref{3-1-3'c11-l}, it yields
	\begin{eqnarray}\notag
	[{\bf M}+ \frac{\tau}{2} (\tilde A(P^{n+1,l})& +&  \tilde A(P^{n+1,l-1}))](\phi^{n+1}(P^{n+1,l})- \phi^{n+1}(P^{n+1,l-1}))
	\\&=&\frac{\tau}{2} ((\tilde A(P^{n+1,l-1})-   \tilde A(P^{n+1,l}))(\phi^{n+1}(P^{n+1,l})+   \phi^{n+1}(P^{n+1,l-1}))).\label{3-1-3'c1-l}
	\end{eqnarray}
	Denoting by $ \delta P:=\phi^{n+1}(P^{n+1,l})- \phi^{n+1}(P^{n+1,l-1})$, from the definitions of ${\bf M}$, $C_{\Phi}^{\tilde A} $ and $\tau_C$ given by \eqref{C4-M-def-2}, \eqref{C-Phi-tildeA-l} and \eqref{def-tauC-l}, respectively,  we have if setting $C_L\ge 1$ and $\tau \le \tau_C$ , the left side of  \eqref{3-1-3'c1-l} can be estimated as follows
	{\small\begin{eqnarray}\notag
		\|[{\bf M}+ \frac{\tau}{2} (\tilde A(P^{n+1,l}) +  \tilde A(P^{n+1,l-1}))]  \delta P\|_{\infty}
		&\ge &\|{\bf M} \delta P\|_{\infty}-\frac{\tau}{2}\| (\tilde A(P^{n+1,l}) +  \tilde A(P^{n+1,l-1})) \delta P \|_{\infty}
		\notag\\&\ge& \frac{1}{4}\min_{k}|\Omega_k|\|  \delta P\|_{\infty}-\frac{\tau}{2} \|  \tilde A(P^{n+1,l})  +  \tilde A(P^{n+1,l-1})  \|_{\infty}\| \delta P \|_{\infty}
		\notag\\&\ge &\frac{1}{4}\min_{k}|\Omega_k|\|  \delta P\|_{\infty}- {\tau}C_{\Phi}^{\tilde A} \| \delta P \|_{\infty}
		\notag\\&\ge &\frac{1}{8}\min_{k}|\Omega_k|\| \delta P \|_{\infty}
		.\label{3.23-l}
		\end{eqnarray}}
	On the other hand, from Lemma \ref{tildeC-P-l} and \eqref{C-Phi-tildeA-l}, the right side of \eqref{3-1-3'c1-l}  can be written as
	\begin{eqnarray}\notag
	&&\frac{\tau}{2}\| (\tilde A(P^{n+1,l-1})-   \tilde A(P^{n+1,l}))(\phi^{n+1}(P^{n+1,l})+   \phi^{n+1}(P^{n+1,l-1}))\|_{\infty}
	\\&\le& {\color{black}\frac{\tau}{2}\| (\tilde A(P^{n+1,l-1})-   \tilde A(P^{n+1,l})\|_{\infty}\|\phi^{n+1}(P^{n+1,l})+   \phi^{n+1}(P^{n+1,l-1}) \|_{\infty}}
	\notag\\&\le& \tau  C_{\Phi}^{\tilde A}\| (\tilde A(P^{n+1,l-1})- \tilde A(P^{n+1,l})\|_{\infty}
	\notag\\&\le&  {\tau} C_{\Phi}^{\tilde A}C_L \|  P^{n+1,l}-P^{n+1,l-1} \|_{\infty}
	.
	\label{3.21-l}
	\end{eqnarray}
Inserting \eqref{3.23-l} and \eqref{3.21-l} into \eqref{3-1-3'c1-l}, it follows that
	\begin{eqnarray*}
		\frac{1}{8}\min_{k}|\Omega_k| ~\|\phi^{n+1}(P^{n+1,l})- \phi^{n+1}(P^{n+1,l-1})\|_{\infty}
		&\le& {\tau} C_L C_{\Phi}^{\tilde A}\| P^{n+1,l}-P^{n+1,l-1}\|_{\infty},
	\end{eqnarray*}
that is
	\begin{eqnarray*}
		\|\phi^{n+1}(P^{n+1,l})- \phi^{n+1}(P^{n+1,l-1})\|_{\infty}
		\le\alpha_{n+1}\| P^{n+1,l}-P^{n+1,l-1} \|_{\infty},~\alpha_{n+1}:=\frac{8  C_L  C_{\Phi}^{\tilde A} }{\min\limits_{k}|\Omega_k|}{\tau}.
	\end{eqnarray*}
Hence $\alpha_{n+1} \in(0,1)$ when $\tau<\tau_C$, which implies \eqref{contraction-l}, i.e. {\color{black}  $\phi^{n+1}$ is contractive on $\mathcal {\tilde C}^{n+1}$}. We complete the proof of this lemma.
\end{proof}

By using the above lemma, we have the following theorem.
\begin{theorem}\label{theorem}
Assume~\eqref{cytj}, \eqref{pn} and \eqref{Fi} hold,  $\tau_{*,n-1}$ is defined in \eqref{FJge0} and $\{P^{n,l}\}$ is the Gummel iteration sequence given by  \eqref{EAFE-non-2-dis}. For any given time step $t=t^{n},n\in\{1,\ldots,N\}$ and any positive constant $\tilde\epsilon$ independent of $\{P^{n,l}\}$, if $\tau<\tau_{*,n-1}$, then the sequence $\{P^{n,l}\}$ converges, i.e. there exists a positive constant $L_n$, such that  $\{P^{n,l},l=0,1,\ldots,L_n\}$ satisfies 
	\begin{align}\label{tol-dis}
	\|P^{n,l+1}-P^{n,l}\|_\infty\leq \tilde\epsilon,
	\end{align}
and the element of $P^{n,l}$ satisfies
	
	\begin{equation}\label{phiP-1}
	0<p^{i,n,l}_{ k}\le {\tilde C}^{n}_k,~\forall k\in\{1,\ldots, n_h\},i\in\{1,2\},~l=0,1,\ldots,L_n.
	\end{equation}
\end{theorem}

\begin{proof}
Since the operator $\phi^{m}(m=1,\ldots, N)$ in  \eqref{EAFE-non-2-dis} and \eqref{UphiJJ}  are the same, we can show Theorem \ref{theorem} by the mathematical induction and following the arguments of Theorem \ref{th-1} and Lemma \ref{lemma2}.
	
	First we prove \eqref{phiP-1}  holds on $t=t^1$, i.e.
	\begin{equation}\label{phiP-2}
 0<p^{i,1,l}_{ k}\le {\tilde C}^{1}_k,~\forall k\in\{1,\cdots, n_h\},i\in\{1,2\},~l=0,1,\cdots.
\end{equation}
From \eqref{pn} and ${\tilde C}_k^{0}:=\|P^0\|_\infty$, it is easy to get \eqref{phiP-2} when $l=0$. Next, we only show  \eqref{phiP-2} for the case $l=1$ as an example, since it is similar for $l\ge 2$. Noting \eqref{pn} is the case $J=0$ in \eqref{pn-J} and using \eqref{Fi}, we get \eqref{FJge0} form Lemma \ref{lemma1}. From~\eqref{EAFE-non-2-dis}, we have~$ P^{1,1}$ satisfies
	\begin{equation}\label{eqJJ-3}
	A (P^{1,0})  P^{1,1} =F^0.
	\end{equation}
	Following the arguments in \eqref{PJge0}, and from \eqref{cytj}
	 and \eqref{FJge0} (i.e. $F^0>0$), we obtain the element of $P^{1,1}$ satisfies
	\begin{equation}\label{PJge0-3}
  p^{i,1,1 }_{k}= ( B^TF^0)_k>0,~ k=1,\ldots, {n_h},i=1,2,
\end{equation}
	Similar to the deduction of \eqref{CJ} 
	,
	using \eqref{cytj} and \eqref{PJge0-3}, we have $E\tilde  A(P^{1,0} )  P^{1,1}> 0$, which combining with  \eqref{eqJJ-3}, \eqref{PJge0-3}, \eqref{EAFE-non-dis}  and \eqref{def-CJ} yields
	\begin{eqnarray*}
C_0=EF^0=E({\bf M}+\tau \tilde  A(P^{1,0})) P^{1,1} >
 E {\bf M} P^{1,1} >  \frac{C_0}{C^{1}_k} p^{i,1,1 }_{k},~\forall k\in\{1,\ldots, n_h\},i\in\{1,2\}.
 \end{eqnarray*}
Then from \eqref{tildeC-C} we have $p^{i,1,1}_{ k}<{C^{1}_k}\le {\tilde C}^{1}_k$. Combining \eqref{PJge0-3}, we get \eqref{phiP-2}.
	
	Next, we show \eqref{tol-dis} when $t=t^1$. From  \eqref{phiP-2}, we know the assumption of Lemma \ref{lemma1} holds when $n=1$. Then from Lemma  \ref{lemma1}, \eqref{EAFE-non-2-dis}  and \eqref{phiP-2}, we obtain
	\begin{eqnarray}\notag
	\| P^{1,l+1} - P^{1,l} \|_{\infty}&=&\|\phi^{1}(P^{1,l})-\phi^{1}(P^{1,l-1})\|_{\infty}
	\notag\\&\le& \alpha_1\| P^{1,l}-  P^{1,l-1} \|_{\infty}
	\notag\\&\le& \alpha_1^{l}\| P^{1,1}-  P^{1,0}\|_{\infty}\le  2\alpha_1^{l} \max_k\{{\tilde C}_k^{1}\},
	\label{ys},
	\end{eqnarray}
	 Hence we have \eqref{tol-dis} when $t=t^1$, and the minimum positive constant which satisfies \eqref{tol-dis} is $L_1=\lceil  \log_{\alpha_{1}}\frac{\tilde\epsilon}{2\max\limits_k\{{\tilde C}_k^{1}\}} \rceil$. Denoting by $P^{1}=P^{1,L_1}$, and taking $l=L_1$ in ~\eqref{phiP-2}, it follows that there exists a positive constant ~$C_{p^{1}}$ such that the element of $ P^{1}$ satisfies
	\begin{equation}\label{p1-33}
C_{p^{1}}:=\min_{k\in\{1,\ldots, n_h\},i= 1,2 }  p^{i,{1} }_k   >0.
\end{equation}

Next, for the time step $t=t^m(m\ge 1)$,  assume there exists a positive constant $L_{m}=\lceil  \log_{\alpha_{m}}\frac{\tilde \epsilon}{2\max\limits_k\{{\tilde C}_k^{m}\}} \rceil$ when
$\tau<\tau_{*,m-1}$,  such that the Gummel iterative sequence $\{P^{m,l},l=0,1,\ldots,L_m\}$ satisfies the convergence codition \eqref{tol-dis} and
	\begin{equation}\label{phiP-m}
 0<p^{i,m,l}_{ k}\le {\tilde C}^{m}_k,~\forall k\in\{1,\ldots, n_h\},i\in\{1,2\},~l=0,1,\ldots,L_m,
\end{equation}
and there is a constant $C_{p^{m}}$ independent of $P^{m}:=P^{m,L_m}$ 
satisfies
	\begin{equation}\label{pJ-3}
	C_{p^{m}}:=\min_{k\in\{1,\ldots, n_h\},i= 1,2 }  p^{i,{m}}_k  >0.
	\end{equation}
	
Now we show
	\eqref{tol-dis}-\eqref{pJ-3} when $m:=m+1$.
	Setting $P^{m+1,0}=P^{m,L_m}$, taking $l=L_m$ in \eqref{phiP-m} and from \eqref{set-JJ-3}, we have
	\begin{equation}\label{phiP-2m+1-0}
0<p^{i,m+1,0}_{ k} =p^{i,m,L_m}_{ k}\le {\tilde C}^{m}_k \le {\tilde C}^{m+1}_k,~\forall k\in\{1,\ldots, n_h\},i\in\{1,2\}.
\end{equation}
Since the assumption of Lemma \ref{lemma1} holds when $J=m$ by \eqref{pJ-3},  combining \eqref{Fi},  using Lemma \ref{lemma1} we have \eqref{FJge0} when $J=m$. Following the deduction of \eqref{phiP-2} for  $l\ge 1$ when $t=t^1$, it yields
	\begin{equation}\label{phiP-2m+1}
 0<p^{i,m+1,l}_{ k}\le {\tilde C}^{m+1}_k,~\forall k\in\{1,\cdots, n_h\},i\in\{1,2\},~l=1,\cdots,
\end{equation}
	which combining with \eqref{phiP-2m+1-0} and  \eqref{phiP-2m+1}, we get
	\begin{equation}\label{phiP-2m+1-01}
 0<p^{i,m+1,l}_{ k}\le {\tilde C}^{m+1}_k,~\forall k\in\{1,\cdots, n_h\},i\in\{1,2\},~l=0,1,\cdots.
\end{equation}
	Further, since \eqref{phiP-2m+1-01} is equivalent to Assumption \eqref{assume} when $n=m+1$, similar to the deduction of ~\eqref{ys}, we have
	\begin{eqnarray}
	\| P^{m+1,l+1} - P^{m+1,l} \|_{\infty} \le  2\alpha_{m+1}^{l} \max_k\{{\tilde C}_k^{m+1}\}. \label{ys-m+1}
	\end{eqnarray}
	 Hence  we get \eqref{tol-dis} on $t=t^{m+1}$ and the minimum positive  constant satisfying
	 \eqref{tol-dis} is $L_{m+1}=\lceil  \log_{\alpha_{m+1}}\frac{\tilde \epsilon}{2\max\limits_k\{{\tilde C}_k^{m+1}\}} \rceil$ .
Setting $P^{m+1}=P^{m+1,L_{m+1}}$ and taking $l=L_{m+1}$ in ~\eqref{phiP-2m+1-01}, we have there exists a constant~$C_{p^{m+1}}$ such that the element of $ P^{m+1}$ satisfies
	\begin{equation}\label{p1-33}
C_{p^{m+1}}:=\min_{k\in\{1,\ldots, n_h\},i= 1,2 }  p^{i,m+1}_k
>0.
\end{equation}
Then using \eqref{phiP-2m+1-01}-\eqref{p1-33}, it follows that the Gummel iterative sequence $\{P^{m+1,l},l=0,1,\ldots,L_{m+1}\}$ satisfies the convergence condition \eqref{tol-dis} when $m:=m+1$, and  ~\eqref{phiP-m} and \eqref{pJ-3} hold.
	We complete the proof of this theorem by the mathematical induction.
\end{proof}

\begin{defi}
The sequence $\{P^{n,l}\}$ is convergent, referring to there existence of $P ^ {n, *} \in R ^ {2n_h} $, such that
\begin{equation} 
\lim\limits_{l\rightarrow\infty}P^{n,l}=P^{n,*}.
\end{equation}
\end{defi}

From Theorem \ref{theorem}, it is easy to obtain the following  corollary.
\begin{coro}
Under the condition of  Theorem \ref{theorem}, for any given time step $t=t^{n},n\in\{1,\ldots,N\}$, the Gummel iteration sequence $\{P^{n,l}\}$ is convergent.
\end{coro}

\subsection{Numerical experiment}
In this subsection, we present a numerical experiments to verify the contraction and convergence of the Gummel iteration. To implement the algorithms, the code is written in
Fortran 90 and all the computations are carried out on the computer with 32-core
187 GB RAM PowerEdge T640.

\begin{exam}
Consider the following dimensionless time-dependent PNP equations in semiconductor area (cf. {\cite{wang2021stabilized}} for the steady-state form):
\begin{equation}\label{cont-new-model}
\left\{
  \begin{array}{llll}
&-\Delta u- ( p-n)=F_1 , & & \text { in } \Omega, \\
&{\partial_t p }-\nabla \cdot\left(\nabla p +c_{\lambda}  p  \nabla u\right) =F_2,  & & \text { in } \Omega, \\
&{\partial_t n}-\nabla \cdot\left(\nabla n-c_{\lambda} n \nabla u\right) =F_3,  & & \text { in } \Omega,\end{array}
\right.
\end{equation}
 with the initial and boudary value conditions
 \begin{equation}\label{boundary-1}
\left\{
\begin{array}{l}
 u=g_u,\text{on}\;\partial\Omega,~\text{for}~t\in(0, T],\vspace{0.5mm}\\
 p =g_{p },~n =g_{n },\;\text{on}\;\partial\Omega,~\text{for}~t\in(0, T],\\
 p ( x,0)=0 ,~n( x,0)=0 ,~\text{for}~  x \in \Omega,
\end{array}
\right.
\end{equation}
where~$\Omega=[-\frac{1}{2}, \frac{1}{2}]^3$ and  $c_{\lambda}=0.179$. 
The initial-boundary condition and the right-hand side functions $F_i,~i=1,2,3$ are given from the following exact solution
\begin{equation*}
\left\{
  \begin{array}{ll}
    u =(1-e^{-t}) \cos(\pi x)\cos(\pi y)\cos(\pi z),&\\
    p =\sin(t)3\pi^2(1+\frac{1}{2}\cos(\pi x)\cos(\pi y)\cos(\pi z)), &  \\
    n =\sin(2t)3\pi^2(1-\frac{1}{2}\cos(\pi x)\cos(\pi y)\cos(\pi z)). &  \\
  \end{array}
\right.
\end{equation*}
\end{exam}
\vskip0.3cm
In our computation, we choose the time step $\tau=h^2$ and set the final time $T=0.25$. The tolerance $\epsilon=1.0\time 10^{-6}$, the maximum iteration $M_{iter}$=500 and the direct solver for linear algebraic system is ~pardiso.

From \eqref{EAFE-non-2-dis} and \eqref{contraction-l}, we have the solution $P^{k,l+1}$ of the Gummel iteration satisfies the contraction on $t=t^{k}$, i.e.
\begin{equation}\label{contraction-2}
\| P^{k,l+1} - P^{k,l} \|_{\infty}\le \alpha^{(l)}_{k}\| P^{k,l}-  P^{k,l-1} \|_{\infty},~l=1,\ldots,{L_k-1},~\alpha^{(l)}_{k}=O(\tau).
\end{equation}
Let $\bar\alpha =\frac{1}{L_k-1}\sum\limits_{l=1}^{L_k-1} \alpha_k^{(l)}$, and  $L_k$ is defined by ~\eqref{phiP-1}. The following Table 1 displays the value of $\bar\alpha$ at different $\tau$ by using the Gummel iteration combining with three FE schemes including standard FE (\ref{fulldis2-FEM})-(\ref{fulldis-FEM}), SUPG (\ref{fulldis-SUPG-temp})-(\ref{fulldis2-SUPG-temp}) and EAFE (\ref{fulldis2-EAFE})-(\ref{fulldis-EAFE}) schemes.
\begin{table} [H]
	\caption{the value of~$\bar \alpha$ with $h=\frac{1}{16}$   }
	\centering
	\begin{tabular}{|c|c|c|c|c|c|c|  }
		\hline
		$\tau$ &  $\bar\alpha$-FEM &   rate &  $\bar\alpha$-SUPG &   rate &  $\bar\alpha$-EAFE &   rate \\
		\hline
		$4h^2$ &  1.63E-1 &  --& 1.62E-1 &  -- &1.63E-1 &  --
		\\
		\hline
		$2h^2$ & 8.55E-2	&  1.91& 8.52E-2	&  1.90& 8.53E-2	&  1.91
		\\
		\hline
		$h^2$ &  4.49E-2 &  1.90& 4.47E-2 &  1.90 &4.48E-2 &  1.90
		\\
		\hline
	\end{tabular}
\end{table}
It is shown from the above table that  $\bar\alpha<1$ and $\tilde{\alpha}$ is  linearly dependent on $\tau$, which means the mapping $\phi^{n+1}$ defined by \eqref{EAFE-non-2-dis} is a contraction and coincides with the theoretical result~\eqref{contraction-l}.

 Next, we present the numerical results of Gummel iteration combining with the three FE schemes.
Figures \ref{FEMl2h1}-\ref{EAFEl2h1} show that the errors between the discrete solution $(u_h,~p_h,~n_h)$ and the weak solution $(u,~p,~n)$ in $L^2$ and $H^1$ norms are second-order and first-order reduction, respectively, which indicates the validity of the three FE schemes used to solve the PNP equations (\ref{cont-new-model}).

\begin{figure}[H]
	\centering
	{
		\begin{minipage}{7cm}
			\centering
			\includegraphics[scale=0.48]
			{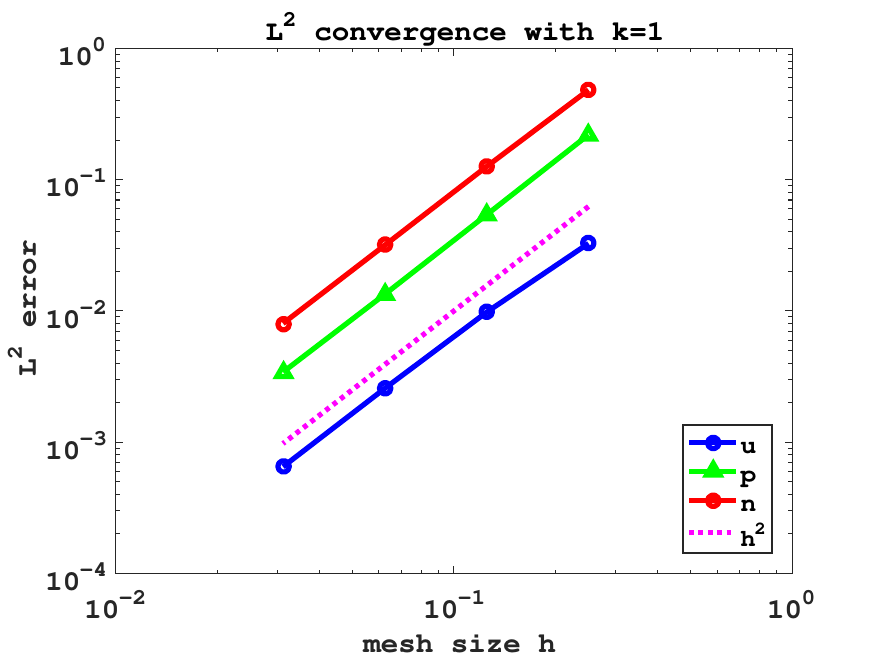}
		\end{minipage}
	}
	{
		\begin{minipage}{7cm}
			\centering
			\includegraphics[scale=0.48]
			{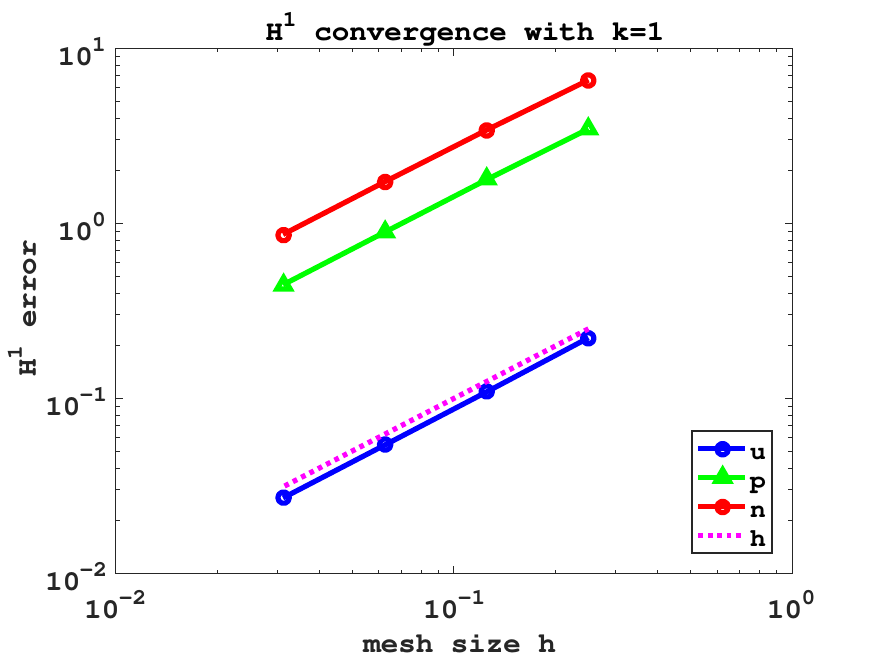}
		\end{minipage}
	}
	
	\caption{h-convergence of FEM scheme  with  $t=0.25$   }
	\label{FEMl2h1}
\end{figure}

\begin{figure}[H]
	\centering
	{
		\begin{minipage}{7cm}
			\centering
			\includegraphics[scale=0.48]
			{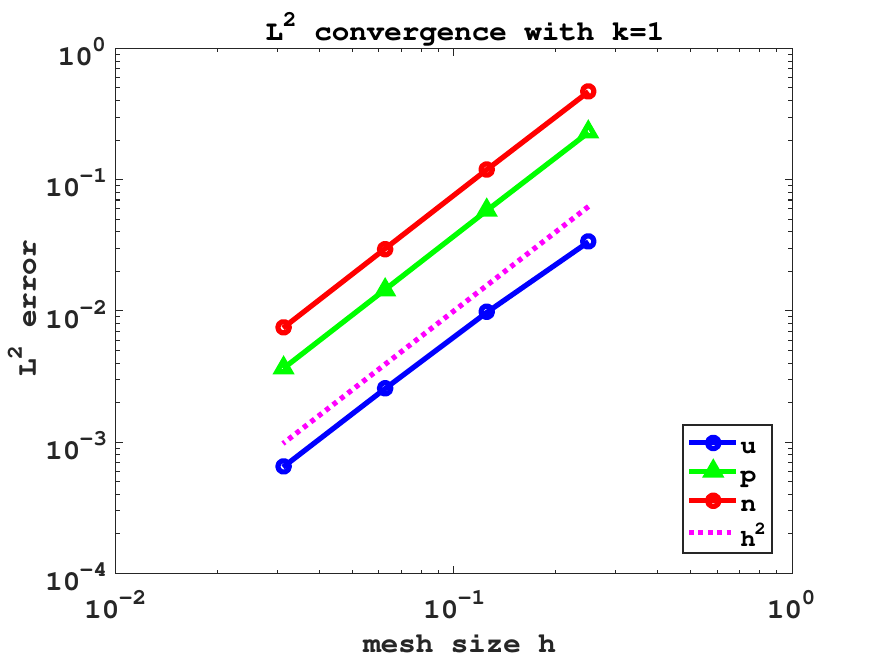}
		\end{minipage}
	}
	{
		\begin{minipage}{7cm}
			\centering
			\includegraphics[scale=0.48]{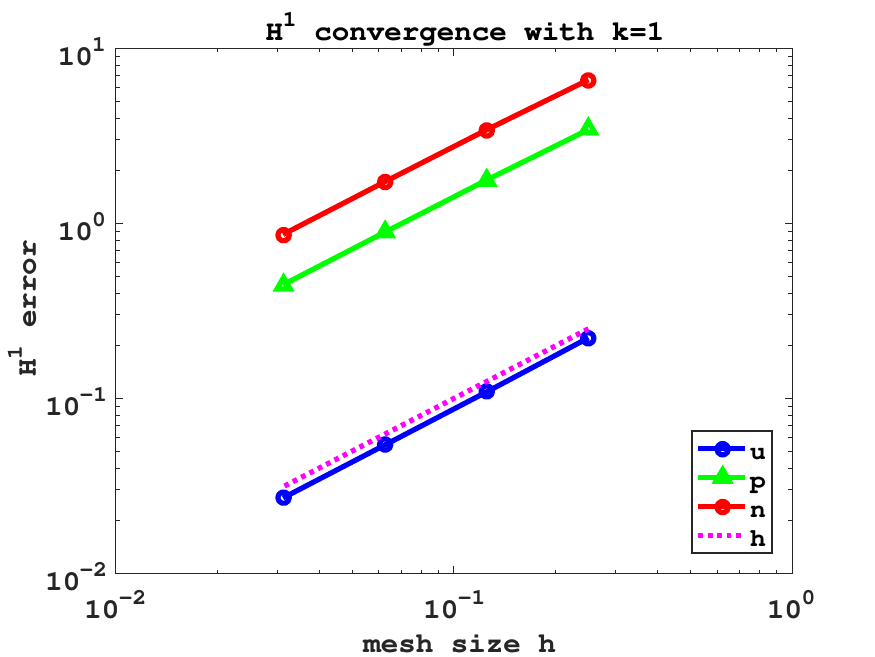}
		\end{minipage}
	}
	
	\caption{h-convergence of SUPG scheme  with  $t=0.25$   }
	\label{SUPGl2h1}
\end{figure}

\begin{figure}[H]
	\centering
	{
		\begin{minipage}{7cm}
			\centering
			\includegraphics[scale=0.48]{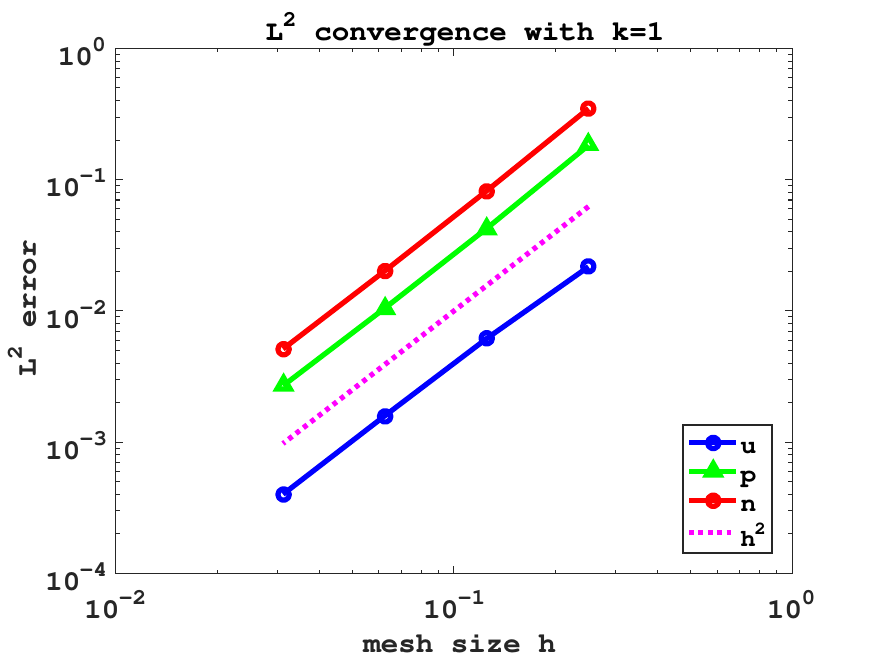}
		\end{minipage}
	}
	{
		\begin{minipage}{7cm}
			\centering
			\includegraphics[scale=0.48]{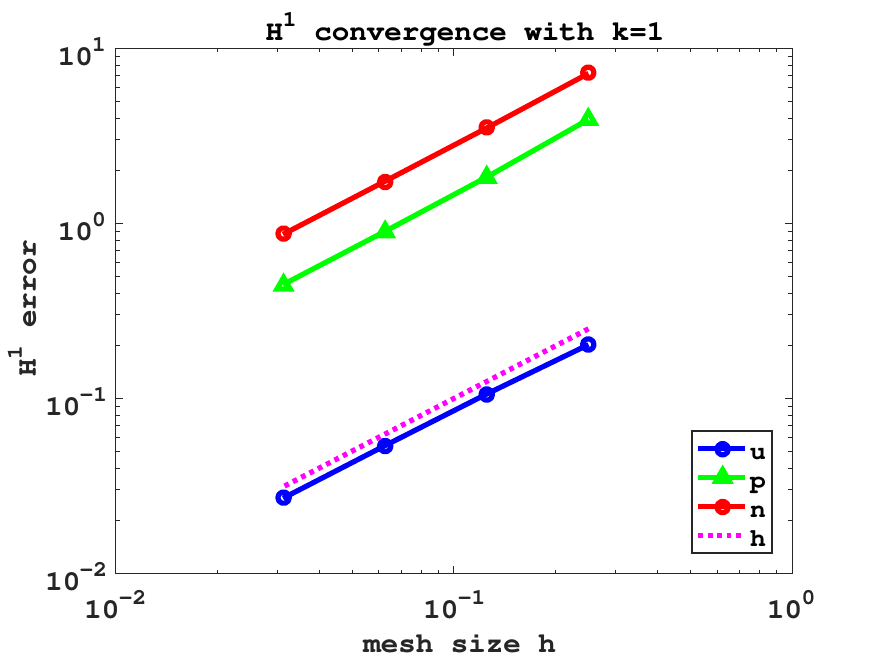}
		\end{minipage}
	}
	
	\caption{h-convergence of EAFE scheme  with  $t=0.25$   }
	\label{EAFEl2h1}
\end{figure}

\vspace{3mm}
\section{Conclusion}
In this paper, we present the theory of the existence of three commonly used FE nonlinear fully discrete solutions and the convergence of the Gummel linearized iteration. The theory of the existence of the solution can be viewed as a framework to FE schemes which only need to satisfy two assumptions. The theory of convergence (including contraction) of the linearized iteration can be easily to generalized to other commonly used iterations such as Picard-Newton iteration. The numerical experiment verifies the result of the theory and also shows the validity of the three FE schemes. Note that although Gummel iteration is effective for some PNP equations, it is not easy to converge for some complex practical PNP problems, for example, the strong convection dominated problem in semiconductor area. In our next upcoming paper, we will design several fast algorithms to improve the convergence efficiency of the iterative algorithm for PNP equations especially for the case with strong convection dominance. \\

\noindent
{\bf Acknowledgement }
  S. Shu was supported by the China NSF (NSFC 12371373). Y. Yang was supported by the China NSF (NSFC 12161026).

\bibliographystyle{unsrt}  
\bibliography{artpnp}
\end{document}